\documentclass{article}

\usepackage{arxiv}
\usepackage[applemac]{inputenc} 		
\usepackage[T1]{fontenc}    		
\usepackage[colorlinks]{hyperref}      
\usepackage{url}            			
\usepackage{amsthm}
\usepackage{booktabs}       		
\usepackage{amsfonts}       		
\usepackage{amsmath}
\usepackage{dsfont}
\usepackage{nicematrix}
\usepackage{nicefrac}       		
\usepackage{microtype}      		
\usepackage{lipsum}				
\usepackage[square,numbers]{natbib}
\usepackage{mathtools}
\usepackage{algorithm}
\usepackage{algorithmicx}
\usepackage{algpseudocode}
\usepackage{graphicx}
\usepackage{indentfirst,latexsym,bm}
\usepackage{subfigure}
\usepackage{amssymb}
\usepackage{xcolor}
\usepackage{comment}
\usepackage{enumitem}
\usepackage{bbm}
\usepackage[subfigure]{tocloft}
\usepackage{pgfplots}

\pgfplotsset{compat=1.18}
\usepackage{tikz}
\usetikzlibrary{arrows.meta,positioning}
\usepgfplotslibrary{fillbetween}
\usetikzlibrary{shapes,decorations}
\usetikzlibrary{fit}
\usetikzlibrary{calc}
\usepackage{comment}
\usepackage{bbding}
\usepackage{pifont}

\newcommand{\xmark}{\color{red}\ding{55}}
\definecolor{lavender}{rgb}{0.9, 0.9, 0.98}
\colorlet{color1}{blue}
\colorlet{color2}{red!50!black}

\definecolor{ivory}{RGB}{218,215,203}

\definecolor{cuhkp}{RGB}{98,56,105} 	
\definecolor{cuhkpl}{RGB}{152,24,147} 	
\definecolor{cuhkb}{RGB}{219,160,1} 	
\definecolor{cuhkbd}{RGB}{178,129,0} 	
\definecolor{cuhkr}{RGB}{88,35,155}  	

\definecolor{blackp}{RGB}{0,0,0} 
\definecolor{redp}{RGB}{255,0,0}
\definecolor{orangep}{RGB}{255,128,0}
\definecolor{brownp}{RGB}{128,77,0}
\definecolor{yellowp}{RGB}{255,230,0}
\definecolor{greenp}{RGB}{128,230,0}
\definecolor{bluep}{RGB}{0,128,255}
\definecolor{purplep}{RGB}{152,24,147}
\definecolor{pinkp}{RGB}{230,0,128}    

\definecolor{turql}{RGB}{53,130,134}
\definecolor{oranl}{RGB}{223,149,86} 
\definecolor{mblue}{RGB}{52,52,246}
\usepackage{hyperref}[6.83]
\hypersetup{
	colorlinks=true,
	frenchlinks=false,
	pdfborder={0 0 0},
	naturalnames=false,
	hypertexnames=false,
	breaklinks,
	allcolors = mblue,
	urlcolor = color2,
}

\RequirePackage[capitalize,nameinlink]{cleveref}[0.19]

\crefname{section}{section}{sections}
\crefname{subsection}{subsection}{subsections}

\Crefname{figure}{Figure}{Figures}
\Crefname{assumption}{Assumption}{Assumptions}

\crefformat{equation}{\textup{#2(#1)#3}}
\crefrangeformat{equation}{\textup{#3(#1)#4--#5(#2)#6}}
\crefmultiformat{equation}{\textup{#2(#1)#3}}{ and \textup{#2(#1)#3}}
{, \textup{#2(#1)#3}}{, and \textup{#2(#1)#3}}
\crefrangemultiformat{equation}{\textup{#3(#1)#4--#5(#2)#6}}%
{ and \textup{#3(#1)#4--#5(#2)#6}}{, \textup{#3(#1)#4--#5(#2)#6}}{, and \textup{#3(#1)#4--#5(#2)#6}}

\Crefformat{equation}{#2Equation~\textup{(#1)}#3}
\Crefrangeformat{equation}{Equations~\textup{#3(#1)#4--#5(#2)#6}}
\Crefmultiformat{equation}{Equations~\textup{#2(#1)#3}}{ and \textup{#2(#1)#3}}
{, \textup{#2(#1)#3}}{, and \textup{#2(#1)#3}}
\Crefrangemultiformat{equation}{Equations~\textup{#3(#1)#4--#5(#2)#6}}%
{ and \textup{#3(#1)#4--#5(#2)#6}}{, \textup{#3(#1)#4--#5(#2)#6}}{, and \textup{#3(#1)#4--#5(#2)#6}}

\crefdefaultlabelformat{#2\textup{#1}#3}

\theoremstyle{plain}
\newtheorem{theorem}{Theorem}[section]
\newtheorem{thm}{Theorem}[section]
\newtheorem{lemma}[thm]{Lemma}
\newtheorem{corollary}[thm]{Corollary}

\newtheorem{remark}[thm]{Remark}
\newtheorem{assumption}[thm]{Assumption}
\theoremstyle{plain}

\newcommand{\mer}{\mathcal{M}}

\newcommand{\R}{\mathbb{R}}
\newcommand{\N}{\mathbb{N}}
\newcommand{\Rn}{\mathbb{R}^d}

\newcommand{\Prob}{\mathbb{P}}
\newcommand{\Exp}{\mathbb{E}}
\newcommand{\vp}{\varphi}

\newcommand{\crit}{\mathrm{crit}}

\newcommand{\sL}{{\sf L}}

\newcommand{\sC}{{\sf C}}
\newcommand{\sD}{{\sf D}}

\newcommand{\sE}{{\sf E}}
\newcommand{\sM}{{\sf M}}

\newcommand{\ti}{\gamma}
\newcommand{\tg}{\Delta}
\newcommand{\tw}{{\sf T}}

\newcommand{\Ti}{\Gamma}

\newcommand{\cO}{\mathcal O}

\newcommand{\cB}{\mathcal B}

\newcommand{\cF}{\mathcal F}
\newcommand{\cE}{\mathcal E}
\newcommand{\cS}{\mathcal{S}}

\newcommand{\cX}{\mathcal{X}}
\newcommand{\cR}{\mathcal{R}}

\newcommand{\sct}[1]{\boldsymbol{#1}}
\newcommand{\se}{\sct{e}}
\newcommand{\sd}{\sct{d}}
\newcommand{\sr}{\sct{r}}

\newcommand{\sV}{\sct{V}}
\newcommand{\sW}{\sct{W}}

\newcommand{\sv}{\sct{v}}
\newcommand{\sg}{\sct{g}}

\newcommand{\scs}{\sct{s}}

\newcommand{\sw}{\sct{w}}
\newcommand{\sx}{\sct{x}}
\newcommand{\sy}{\sct{y}}
\newcommand{\sz}{\sct{z}}

\newcommand{\SGD}{\mathsf{SGD}}
\newcommand{\RR}{\mathsf{RR}}
\newcommand{\SGDM}{\mathsf{SGDM}}
\newcommand{\SHB}{\mathsf{SHB}}

\newcommand{\SNAG}{\mathsf{SNAG}}

\newcommand{\nx}{\tilde{\sct{x}}}
\newcommand{\mparam}{\nu}

\definecolor{purp}{RGB}{152,24,147}
\definecolor{bluep}{RGB}{0,128,255}
\definecolor{redp}{RGB}{255,0,0}
\definecolor{orangep}{RGB}{255,128,0}

\definecolor{OxfordBlue}{rgb}{0,0.106,0.329}   
\definecolor{UMRed}{rgb}{0.73,0.09,0.19}   
\definecolor{CUBrown}{RGB}{152,95,42}   
\definecolor{LightBrown}{RGB}{219,199,181}   
\definecolor{MyBlue}{RGB}{81,168,221}   
\definecolor{LightBlue}{RGB}{229,232,247}   
\definecolor{MyGreen}{RGB}{165,222,228}   
\definecolor{LightGreen}{RGB}{228,245,247} 

\definecolor{McBlack}{RGB}{39,37,31} 
\definecolor{McRed}{RGB}{218,41,28} 
\definecolor{McYellow}{RGB}{255,199,44}

\newcommand{\iprod}[2]{\langle #1, #2 \rangle}

\title{Convergence of SGD with momentum in the nonconvex setting: A time window-based analysis\thanks{{Andre Milzarek was partly supported by the National Natural Science Foundation of China (Foreign Young Scholar Research Fund Project) under Grant No$.$ 12150410304, by the Shenzhen Science and Technology Program (No$.$ RCYX20221008093033010), and by the Guangdong Provincial Key Laboratory of Mathematical Foundations for Artificial Intelligence (2023B1212010001).}}}
\date{} 					

\author{
Junwen Qiu \\
	Industrial Systems Engineering \& Management \\
	National University of Singapore\\ 
    \texttt{jwqiu@nus.edu.sg} \\
    \And
    Bohao Ma \\
	School of Data Science (SDS) \\
	The Chinese University of Hong Kong, Shenzhen \\ 
	\texttt{bohaoma@link.cuhk.edu.cn} \\
    \And
	Andre Milzarek\\
	School of Data Science (SDS) \\
	The Chinese University of Hong Kong, Shenzhen \\ \texttt{andremilzarek@cuhk.edu.cn} \\	
}

\begin{document}
	\maketitle	
    \vspace{-8mm}
\begin{abstract}
The stochastic gradient descent method with momentum (SGDM) is a common approach for solving large-scale and stochastic optimization problems. Despite its popularity, the convergence behavior of SGDM remains less understood in nonconvex scenarios. This is primarily due to the absence of a sufficient descent property and challenges in simultaneously controlling the momentum and stochastic errors in an almost sure sense. To address these challenges, we investigate the behavior of SGDM over specific time windows, rather than examining the descent of consecutive iterates as in traditional studies. This time window-based approach simplifies the convergence analysis and enables us to establish the iterate convergence result for SGDM under the {\L}ojasiewicz property. We further provide local convergence rates which depend on the underlying {\L}ojasiewicz exponent and the utilized step size schemes.
\end{abstract}
	

\tableofcontents

\section{Introduction}
In this work, we consider the general optimization problem
\begin{equation} \label{eq:prob-exp}
{\min}_{x\in\Rn}~f(x),
\end{equation}
where $f : \Rn \to \R$ is a smooth but potentially \emph{nonconvex} mapping. 
Many problems in stochastic approximation and optimization take the form \eqref{eq:prob-exp}. Here, the objective function $f$ typically corresponds to some data-driven predictive learning or expected risk task, such as, e.g., $f(x)=\Exp_{\xi\sim\Xi}[F(x,\xi)]$, where $(\Xi,{\cal H},\Prob)$ is an underlying probability space, \cite{bottou2018optimization,chung1954stochastic,ghadimi2013stochastic,kushner2003stochastic,robbins1951stochastic}. Modern machine learning and deep learning problems frequently serve as important examples of such applications, \cite{bottou2018optimization,cutkosky2020momentum,krizhevsky2012imagenet,sutskever2013importance,tseng1998incremental}. 

Stochastic gradient descent ($\SGD$) \cite{robbins1951stochastic} is perhaps one of the most classical and successful methods in dealing with \eqref{eq:prob-exp}. In practice, momentum-based techniques are broadly used to enhance the performance of $\SGD$,  \cite{rumelhart1986learning,sutskever2013importance,cutkosky2020momentum,liu2022almost,liu2020improved,sebbouh2021almost}. The stochastic gradient descent method with momentum ($\SGDM$) generates iterates $\{x^k\}_k$ through the following mechanism: Given $x^0\in\Rn$ and $x^1 = x^0$, the update of $\SGDM$ reads as
\begin{equation}
	\label{eq:update-SGDM-1}
 \left[
    \begin{aligned}
        \tilde x^k &= x^k + \mparam(x^k - x^{k-1}),\\
        g^k & = \nabla f(\tilde x^k) - e^k, \\  
        x^{k+1} & = x^k - \alpha_k g^k + \lambda(x^k-x^{k-1}),
    \end{aligned}
    \right.
\end{equation}
where $e^k$ is the stochastic error (or noise), $\alpha_k>0$ is the step size and  $\mparam \geq 0$, $\lambda\in[0,1)$ are the momentum parameters. 
The update rule \eqref{eq:update-SGDM-1} is a general framework encompassing various momentum-based optimization techniques. In the case $\nu=0$, \eqref{eq:update-SGDM-1} reduces to the stochastic variant of Polyak's momentum (or heavy ball) method ($\SHB$), \cite{polyak1964some, sutskever2013importance,cutkosky2020momentum,liu2022almost,LIU202327}, while $\nu = \lambda$ corresponds to the stochastic version of Nesterov's accelerated gradient method ($\SNAG$), \cite{nesterov1983method,sutskever2013importance,LIU202327}. Despite the rich body of work on $\SGD$, the theoretical analysis of its momentum variants \eqref{eq:update-SGDM-1}, particularly in nonconvex settings, remains fairly limited. This paper aims to address this gap by establishing comprehensive asymptotic convergence guarantees for $\SGDM$ in the context of nonconvex optimization.

\subsection{Related Work} 
Momentum-based methods have been reported to speed up the training of neural networks, \cite{tseng1998incremental,krizhevsky2012imagenet,sutskever2013importance,goodfellow2016deep}, 
and most learning libraries, including TensorFlow \cite{abadi2016tensorflow} and PyTorch \cite{paszke2019pytorch}, provide built-in support for such algorithms. Among these, $\SHB$ and $\SNAG$ stand out as widely utilized, standard momentum approaches.

In the following, we discuss existing theoretical results for (stochastic) momentum methods with a focus on the nonconvex setting. In the deterministic case, the first convergence result for Polyak's momentum method for nonconvex problems dates back to \cite{zavriev1993heavy} showing $\nabla f(x^k) \to 0$. 
When $f$ satisfies the {\L}ojasiewicz inequality, iterate convergence (i.e., $x^k\to x^*\in\crit(f)$) of Polyak's momentum and Nesterov's accelerated gradient methods has been established in, e.g., \cite{OchCheBroPoc14,Csaba2021,josz2023convergence}.

In the stochastic case, a universal and common assumption is that the stochastic gradient, $g^k$, is an unbiased approximation of the true gradient, $\nabla f(\tilde x^k)$, and satisfies a bounded variance condition (cf$.$ \Cref{as:sgd}). In this setting, Liu et al$.$, \cite{LIU202327} have proved $\Exp[\|\nabla f(\sx^k)\|] \to 0$ for $\SGDM$.
In the case $\nu=0$, $\lambda\in[0,1)$, Liu and Yuan \cite{liu2022almost,liu2024almost} showed $\nabla f(\sx^k) \to 0$ almost surely (a.s.) for $\SHB$ when $f$ is $\sL$-smooth. Furthermore, if $f$ is convex, the authors in \cite{liu2022almost,liu2024almost} derived $\sx^k\to \sx^*\in\crit(f)$ and verified the asymptotic rate $f(\sx^k)-f^*= \cO(k^{-\frac13+\varepsilon})$, $\varepsilon>0$ a.s..
Several studies have analyzed the convergence of a specific variant of $\SHB$ with specialized momentum parameters, i.e.,
\begin{equation}
\label{eq:special-SHB}
x^{k+1} = x^k - \alpha_k (\nabla f(x^k) -e^k) + \lambda_k(x^k-x^{k-1})\quad \text{with}\quad \lambda_k\to0\quad \text{or} \quad \lambda_k\to1.
\end{equation}
For coercive functions, Gadat et al. \cite{gadat2018stochastic} established almost sure convergence in the sense $\nabla f(\sx^k)\to 0$ when $\lambda_k \to 1$. In the convex case, Sebbouh et al. \cite{sebbouh2021almost} proved almost sure convergence of the iterates, $\sx^k \to \sx^*\in\crit(f)$, if $\lambda_k\to1$. Dereich and Kassing \cite{dereich2024} extended these convergence results to functions satisfying the {\L}ojasiewicz inequality, showing almost sure convergence of the iterates for a momentum approach similar to \eqref{eq:special-SHB}. Specifically, the update in \cite{dereich2024} is based on the gradient $\nabla f(x^{k-1})$ (instead of $\nabla f(x^k)$) and the condition $\lambda_k\to1$ is again required in the analysis. 
Very recently, Liu and Yuan \cite[Theorem 11]{liu2024almost} established global convergence results ($\nabla f(\sx^k) \to 0$ a.s.) for $\SNAG$ (i.e., when $\nu=\lambda$ in \eqref{eq:update-SGDM-1}). To the best of our knowledge, iterate convergence guarantees ($\sx^k\to \sx^*\in\crit(f)$ a.s.) are not known for $\SHB$ and $\SNAG$ when applied to nonconvex objectives. Our work establishes comprehensive convergence guarantees for $\SGDM$ in a general nonconvex setting for arbitrary momentum parameters $\lambda \in [0,1)$ and $\nu \geq 0$. Specifically, our analysis provides both global and iterate convergence results, along with almost sure convergence rates.


\begin{table}[t]
\centering
{\footnotesize

\NiceMatrixOptions{cell-space-limits=2pt}
\begin{NiceTabular}{|p{1.5cm}|p{1.5cm}p{2cm}|p{2cm}p{2cm}p{2cm}|p{1cm}|}%
 [ 
   code-before = 
    \rectanglecolor{lavender!40}{3-2}{18-2}
    \rectanglecolor{lavender!40}{3-4}{18-4}
    \rectanglecolor{lavender!40}{3-6}{18-6}
 ]
\toprule
\Block[c]{2-1}{\textbf{Alg.}}&\Block[c]{1-2}{\textbf{Conditions}} && \Block{1-3}{\textbf{Almost Sure Convergence}}  &&& \Block{2-1}{\textbf{Ref.}}   \\\Hline   & \Block{1-1}{noncvx.${}^{\textcolor{blue}{\text{(a)}}}$} & \Block{1-1}{bounded iter. \emph{not} required} & \Block{1-1}{glob. conv. $\nabla f(\sx^k)\to0$} & \Block{1-1}{iter. conv. $\sx^k \to \sx^*$}  & \Block{1-1}{conv. rate} & \\ \Hline  \Block[l]{1-1}{$\SGD$}& \Block{1-1}{\Checkmark} &\Block{1-1}{\xmark} & \Block{1-1}{\Checkmark} & \Block{1-1}{\Checkmark} & \Block{1-1}{\Checkmark}  & \Block{1-1}{\cite{tadic2015convergence}} \\
\Block[l]{1-1}{$\SHB^{\textcolor{blue}{\text{(b)}}}$} & \Block{1-1}{{\Checkmark}\,/\,{\xmark}} & \Block{1-1}{{\Checkmark}\,/\,{\Checkmark}} &  \Block{1-1}{{\Checkmark}\,/\,{\Checkmark}} & \Block{1-1}{{\xmark}\,/\,{\Checkmark}} & \Block{1-1}{{\xmark}\,/\,{\xmark}}  & \Block{1-1}{\cite{sebbouh2021almost}} \\
\Block[l]{1-1}{$\SHB$/$\SNAG$} & \Block{1-1}{{\Checkmark}\,/\,{\xmark}} & \Block{1-1}{{\Checkmark}\,/\,{\Checkmark}} &  \Block{1-1}{{\Checkmark}\,/\,{\Checkmark}} & \Block{1-1}{{\xmark}\,/\,{\Checkmark}} & \Block{1-1}{{\xmark}\,/\,{\xmark}}  & \Block{1-1}{\cite{liu2022almost,liu2024almost}} \\
\Block[l]{1-1}{$\SHB^{\textcolor{blue}{\text{(c)}}}$} & \Block{1-1}{\Checkmark} & \Block{1-1}{\xmark} &   \Block{1-1}{\Checkmark} & \Block{1-1}{\Checkmark} & \Block{1-1}{\xmark}  & \Block{1-1}{\cite{dereich2024}} \\
\Block[l]{1-1}{$\SGDM^{\textcolor{blue}{\text{(d)}}}$} & \Block{1-1}{\Checkmark} & \Block{1-1}{\Checkmark} &   \Block{1-1}{\textcolor{black!50}{\Checkmark}} & \Block{1-1}{\xmark} & \Block{1-1}{\xmark}  & \Block{1-1}{\cite{LIU202327}} \\
\Block[l]{1-1}{$\SGDM$} & \Block{1-1}{\Checkmark} & \Block{1-1}{\Checkmark} &  \Block{1-1}{\Checkmark} & \Block{1-1}{\Checkmark} & \Block{1-1}{\Checkmark}  & \Block{1-1}{Ours} \\
\bottomrule
\end{NiceTabular}
}
\caption{Overview of convergence results for stochastic (momentum) methods.
\endgraf
\setlength{\parindent}{1ex} \setlength{\baselineskip}{11pt}
\noindent \hspace*{1pt} ${}^{\textcolor{blue}{\text{(a)}}}$ {\footnotesize Results for strongly convex objective functions are not considered.} \\
\hspace*{1pt} ${}^{\textcolor{blue}{\text{(b)}}}$ {\footnotesize This work requires $\{\lambda_k\}_k$ defined in \eqref{eq:special-SHB} to converge to $1$. The result on iterate convergence is based on \emph{convexity} of $f$ (similarly for the entry on $\SHB$/$\SNAG$ by \cite{liu2022almost}).} \\ \hspace*{1pt} ${}^{\textcolor{blue}{\text{(c)}}}$ {\footnotesize The results hold for the special case where the sequence of momentum parameters $\{\lambda_k\}_k$ defined in \eqref{eq:special-SHB} to converge to $1$.}
\\ \hspace*{1pt} ${}^{\textcolor{blue}{\text{(d)}}}$ {\footnotesize  $\Exp[\|\sg^k\|^2]$ is required to be bounded; convergence is shown in expectation $\Exp[\|\nabla f(\sx^k)\|]\to 0$.}
\endgraf}
\vspace{-2ex}
\end{table}

\subsection{Contributions}
The analysis of $\SGDM$ is challenging due to the inherent entanglement of stochastic errors and the momentum mechanism, which complicates the derivation of convergence properties. Even in the deterministic setting, momentum methods typically do not monotonically decrease the objective function across iterations. This effect is amplified in $\SGDM$, where individual trajectories can exhibit vastly different behavior. To address these challenges, we employ a twofold strategy: \\[1mm]
\indent\emph{Time window-based analysis.} We leverage time window techniques (see \Cref{subsec:time win}) 
to effectively estimate the aggregation of stochastic errors (\Cref{lem:err_estimate}) and to provide iterate bounds for $\SGDM$ (\Cref{lem:iterate-bound}) in an almost sure sense. \\[1mm]
\indent\emph{Auxiliary iterates and merit function.} In \Cref{subsec:merit func}, we employ an auxiliary iterate sequence $\{z^k\}_k\subset\Rn$, coupled with a carefully designed merit function $\mer$. This strategic pairing allows disentangling the momentum term from the main dynamics of $\SGDM$. Together with the time window techniques, this enables us to establish an approximate descent-type property (\Cref{lem:approx-desent}) for $\SGDM$. \\[1mm]
\indent These tools allow us to show novel convergence results for momentum methods in the stochastic setting. Our key contributions are summarized as follows:
\begin{itemize}[leftmargin=4.5ex]
    \item In the nonconvex setting, we show the convergence of the function and gradient values for the generic stochastic momentum mechanism --- $\SGDM$. In particular, we prove that $\{f(\sx^k)\}_k$ converges and $\|\nabla f(\sx^k)\|\to0$ almost surely (\Cref{Prop:convergence}).  
    \item Leveraging the {\L}ojasiewicz inequality --- a mild assumption on the local geometry of $f$ --- we establish the almost sure convergence of the iterates generated by $\SGDM$ to a stationary point of $f$. This seems to be the first iterate convergence guarantee for stochastic momentum methods in a nonconvex setting  with momentum parameters being chosen freely from $\lambda \in[0,1)$ and $\nu\geq 0$\footnote{{In practice, the momentum parameter is a fixed number, typically, $\lambda=0.9$ rather than $\lambda_k \to 0$ or $1$ for $\SHB$, see, \cite{sutskever2013importance,rumelhart1986learning} and \cite[Chapter 8.3.2]{goodfellow2016deep}.}}. 
    We prove that the iterates generated by $\SGDM$ converge to some stationary point almost surely without requiring the ubiquitous ``a priori bounded iterate'' assumption, $\limsup_{k\to\infty} \|\sx^k\| < \infty$ a.s.\footnote{{This assumption seems prevalent in the study of stochastic algorithms under the {\L}ojasiewicz inequality \cite{tadic2015convergence,li2023convergence,dereich2024}.}}, or convexity of $f$. Instead we work with the weaker condition $\liminf_{k\to\infty} \|\sx^k\| < \infty$ (a.s.) which is standard in \emph{deterministic} {\L}ojasiewicz inequality-based analyses \cite{absil2005convergence,AttBol09}. This extends the theoretical guarantees for $\SGDM$ to a wider spectrum of optimization problems.
    \item Beyond iterate convergence, we further provide the asymptotic convergence rates of $\SGDM$ for general step sizes (\Cref{thm:convergence-rate}) and for polynomial step sizes (\Cref{coro:local-rate}) --- depending on the underlying {\L}ojasiewicz exponent. The obtained rates improve the existing results in the convex \cite{sebbouh2021almost,liu2022almost} and nonconvex setting \cite{tadic2015convergence} (for $\SGD$) and can match the rates in the strongly convex setting. 
\end{itemize}

\section{Modeling the Stochastic Process and Assumptions}
We assume that there is a sufficiently rich probability space $(\Omega,\mathcal F,\{\mathcal F_k\}_k,\Prob)$ that allows us to model the stochastic components of $\SGDM$ in a unified way. We will use bold letters $\sx^k$, $\sg^k$, etc$.$ to represent random variables while lowercase letters are typically reserved for realizations of a random variable, $x^k = \sx^k(\omega)$, or deterministic parameters. Hence, each approximation of $\nabla f(x^k)$ is understood as a realization of a random vector $\sg^k : \Omega \to \Rn$. For $\lambda \in [0,1)$, $\mparam \geq 0$, $\SGDM$ generates a stochastic process $\{\sx^k\}_k$ via
\begin{equation}
	\label{algo:sgdm}
 \left[
     \begin{aligned}
        \nx^k &= \sx^k + \mparam(\sx^k - \sx^{k-1}),\\
        \sg^k & = \nabla f(\nx^k) - \se^k, \\  
        \sx^{k+1} & = \sx^k - \alpha_k \sg^k + \lambda(\sx^k-\sx^{k-1}),
    \end{aligned}
    \right.
\end{equation}
where $\se^k$ represents the stochastic errors and $\sx^1 = \sx^0= x^0$ are deterministic initial points. We work with the natural filtration ${\cal F}_k:=\sigma(\sx^0,\sx^1,\ldots,\sx^k)$ (i.e., each $\sx^k$ is $\mathcal F_{k}$-measurable). Next, we introduce our main assumptions on the stochastic errors.
\begin{assumption} \label{as:sgd}
	Given the probability space $(\Omega,{\cal F},\{{\cal F}_k\}_k,\Prob)$, the errors $\{\se^k\}_k$ are assumed to satisfy $\Exp[\se^k \mid \mathcal F_{k}]=0$ and $\Exp[\|\se^k\|^2 \mid \mathcal F_{k}] \leq \sigma^2$ (a.s.) for all $k$.
\end{assumption}

This assumption is standard in the analysis of stochastic methods \cite{borkar2009stochastic,bottou2018optimization,cutkosky2020momentum,milzarek2023convergence}. Next, we impose the following assumption on the objective function $f$. 

\begin{assumption} \label{Assumption:1} 
The function $f$ is bounded from below by some $\bar f\in\R$ and the gradient $\nabla f$ is Lipschitz continuous with parameter $\sL>0$.
\end{assumption}

This assumption is ubiquitous among studies on the convergence of first-order methods, see, \cite{AttBolSva13,BolSabTeb14,li2015global,bottou2018optimization}.
Below, we formally introduce the {\L}ojasiewicz inequality, a mild assumption on the local geometry of the objective function. It plays a key role in establishing the iterate convergence of $\SGDM$.

\begin{assumption}
\label{Assumption:2} 
    The function $f$ satisfies the {\L}ojasiewicz inequality on $\crit(f)$, i.e., for every $x^* \in \crit(f)$, there are $\eta\in (0,\infty]$ and  a neighborhood $U(x^*)$ of $x^*$ 
    such that
  \[\hspace{-2ex}\|\nabla{f}(x)\|\geq \sC_f|f(x)-f(x^*)|^\theta \quad \forall~x \in U(x^*) \cap \{x \in \Rn: 0 < |f(x) - f(x^*)| < \eta\},\]
where $\sC_f = \sC_f(x^*) >0$ and $\theta = \theta(x^*) \in[\frac12,1)$ denotes the {\L}ojasiewicz exponent of $f$ at $x^*$.
\end{assumption}

One main feature of the {\L}ojasiewicz inequality is that it holds universally for subanalytic and semialgebraic functions \cite{lojasiewicz1965ensembles,kur98,absil2005convergence}. Moreover, a broad class of problems arising in practice and applications satisfy this property, see \cite[Section 4]{AttBolRedSou10} and \cite[Section 5]{BolSabTeb14}.  

\section{Time Window Techniques}\label{sec:time win}
In this section, we introduce a time window-based analysis technique, which is tailored to facilitate the iterate convergence analysis of stochastic methods. We start with a toy example to motivate why such tool is needed in the nonconvex stochastic setting. We formally define the time window in \Cref{subsec:time win}. The time window introduces a new ``scaling'', which is then used to derive an approximate descent-type property for $\SGDM$ in \Cref{subsec:merit func}.

\subsection{Limitations of Classical Approaches}\label{subsec:toy-example} Conventional convergence analyses under the {\L}ojasiewicz inequality are typically based on a verification of the \emph{finite length property}, i.e., $\sum_{k=1}^\infty \|x^{k+1} - x^{k}\| <\infty$,
\cite{absil2005convergence,AttBolRedSou10,BolSabTeb14,li2015global,li2023convergence,OchCheBroPoc14}. This readily implies that $\{x^k\}_k$ is a Cauchy sequence and thus, the iterates $\{x^k\}_k$ have to converge. However, we will now illustrate that this property generally does \emph{not hold} for stochastic methods.

Let us consider the two-dimensional case and let the objective function $f:\R^2\to\R$ and the stochastic error $\se^k:\Omega \to \R^2$ be given by 
\[
\begin{aligned}
    f(x)=f(x_1,x_2) := \sin(x_1)\quad \text{and} \quad \se^k:=\begin{cases}
        [0,1]^\top & \text{w.p. }50\%,\\
        [0,-1]^\top & \text{w.p. }50\%.\\
    \end{cases}
\end{aligned} 
\]
The function $f$ is Lipschitz smooth and analytic (thus, the {\L}ojasiewicz inequality holds, \cite{lojasiewicz1965ensembles}), and we have $\Exp[\se^k] = 0$, $\Exp[\|\se^k\|^2] = 1$, and $\iprod{\nabla f(x)}{\se^k}=0$ for all $x\in\R^2$ and $k \geq 1$. By the update rule of $\SGD$, i.e., $\sx^{k+1} = \sx^k - \alpha_k (\nabla f(\sx^k) - \se^k)$, we obtain
\[
    \|\sx^{k+1} - \sx^k\| = \alpha_k\|\nabla f(\sx^k) - \se^k\| =   \alpha_k\sqrt{\|\nabla f(\sx^k)\|^2 + \|\se^k\|^2} \geq \alpha_k\|\se^k\| = \alpha_k.
\]
Under the usual condition $\sum_{k=1}^\infty \alpha_k=\infty$, \cite{borkar2009stochastic,bottou2018optimization,chung1954stochastic,robbins1951stochastic}, it then follows $\sum_{k=1}^\infty \|\sx^{k+1} - \sx^k\| = \infty$ a.s.. Consequently, we can not mimic the conventional, deterministic analysis route to prove iterate convergence of stochastic methods in the nonconvex case. \\[1mm]
\noindent\textbf{Discussion.} The finite length of $\{x^k\}_k$ is not necessary for convergence. In fact, it is sufficient to show that $\{x^k\}_k$ is a Cauchy sequence. To achieve this, we introduce an infinite and increasing subsequence $\{\ti_k\}_k\subset \N$ and define $\Ti_k:=\{t\in\N : \ti_k < t \leq \ti_{k+1}\}$. Our goal is to verify that
\begin{equation*}
    {\sum}_{k=1}^\infty \|x^{\ti_{k+1}} - x^{\ti_{k}}\|<\infty \quad \text{and} \quad {\max}_{t\in\Ti_k} \|x^{t} - x^{\ti_k}\| \to 0\quad \text{as\quad $k\to\infty$}.
\end{equation*}
Hence, for any given $\varepsilon>0$, there exists $k_0 \in \N$ such that for all $k \geq k_0$, it holds that ${\sum}_{t=k}^\infty \|x^{\ti_{t+1}} - x^{\ti_{t}}\|<\frac{\varepsilon}{3}$ and ${\max}_{t\in\Ti_k} \|x^{t} - x^{\ti_k}\| < \frac{\varepsilon}{3}$. Furthermore, for all $n>m \geq \gamma_{k_0}$, there are integers $k_2\geq k_1 \geq k_0$ such that $m\in\Ti_{k_1}$ and $n\in\Ti_{k_2}$. Thus, we obtain
\[
\begin{aligned}
    \|x^n - x^m\| &\leq \|x^n - x^{\ti_{k_2}}\| + \|x^{\ti_{k_2}} - x^{\ti_{k_1}}\| + \|x^m - x^{\ti_{k_1}}\|  \\
    & \leq \max_{t\in\Ti_{k_1}} \|x^{t} - x^{\ti_{k_1}}\| + \max_{t\in\Ti_{k_2}} \|x^{t} - x^{\ti_{k_2}}\| + {\sum}_{t=k_1}^\infty \|x^{\ti_{t+1}} - x^{\ti_{t}}\| < \varepsilon,
\end{aligned}
\]
which indicates that $\{x^k\}_k$ is a converging Cauchy sequence. The specific construction of $\{\gamma_k\}_k$ is tied to ``time windows'' and will be presented in the following subsection. The proof strategy sketched here will be formally explored in \Cref{sec:convergence analysis}. 

\subsection{The Time Window}\label{subsec:time win}
Inspired by \cite{tadic2015convergence} and classical stochastic approximation literature \cite{borkar2009stochastic,kushner2003stochastic,ljung1977analysis}, we study the behavior of $\SGDM$ using the natural time scales \[\tg_{k,k}=0 \quad \text{and} \quad \tg_{k,n} = {\sum}_{i=k}^{n-1}\alpha_i \quad \text{for $k<n$}.
\] 
Specifically, similar to \cite{tadic2015convergence}, let us define the mapping \[ \varpi:\N\times\R_+ \to \N, \quad \varpi(k,\tw):=\max \left\{k+1,\sup\{n \geq k: \tg_{k,n} \leq \tw\} \right\}. \] 
Here, $\tw\in\R_+$ is referred to as a \textit{time window} and the associated \textit{time indices} $\{\ti_k\}_k$ are defined recursively via: 
\[\ti_1 = 1 \quad \text{and} \quad  \ti_{k+1}:=\varpi(\ti_k,\tw) \quad \text{for $k\geq 1$.}\]
Based on the time window and indices, we define the collection $\Ti_k$ of indices within the $k$-th time window as $\Ti_k:=\{t\in\N: \ti_k < t \leq \ti_{k+1}\}$.

\begin{figure}[t]
\centering
\begin{tikzpicture}[scale = .95]
\begin{axis}[
 at={(0,0)},
   width=18.8cm,
    height=4cm, 
    domain=0.8:11.2,
    samples=100,
    ymin=-0.5, ymax=0.82,
    xmin=0.8, xmax=11.2,
    legend pos=north east,
    ytick={1},
    xtick={1, 2, 3, 4, 5, 6, 7, 8, 9, 10, 11},
    xticklabels={$\alpha_1$,$\alpha_2$,$\alpha_3$,$\alpha_4$,$\alpha_5$,$\alpha_6$,$\alpha_7$,$\alpha_8$,$\alpha_9$,$\alpha_{10}$,$\alpha_{11}$},
    restrict y to domain=-0.5:1,
   tick align=outside,
    tickpos=left, 
    tick style={draw=black!50!, line cap=round, major tick length=4pt, minor tick length=2pt,  thick}, 
    axis line style={draw=black!50,  thick}, 
    every axis plot/.append style={thick},
    legend style={at={(.98,.94)}, draw=black, thick, font=\footnotesize},
    axis background/.style={fill=red!0},
    every axis/.append style={font=\color{black}},
]
\addplot[opacity=0.8, black, line width=2pt] {1/x};
    
\addplot[green!60!blue, mark=*, only marks, mark size=1.5pt] coordinates {(1,1)};
\addplot[green!60!blue, mark=*, only marks, mark size=1.5pt] coordinates {(2,1/2)};
\addplot[green!60!blue, mark=*, only marks, mark size=1.5pt] coordinates {(3,1/3)};
\addplot[green!60!blue, mark=*, only marks, mark size=1.5pt] coordinates {(4,1/4)};
\addplot[green!60!blue, mark=*, only marks, mark size=1.5pt] coordinates {(5,1/5)};
\addplot[green!60!blue, mark=*, only marks, mark size=1.5pt] coordinates {(6,1/6)};
\addplot[green!60!blue, mark=*, only marks, mark size=1.5pt] coordinates {(7,1/7)};
\addplot[green!60!blue, mark=*, only marks, mark size=1.5pt] coordinates {(8,1/8)};
\addplot[green!60!blue, mark=*, only marks, mark size=1.5pt] coordinates {(9,1/9)};
\addplot[green!60!blue, mark=*, only marks, mark size=1.5pt] coordinates {(11,1/11)};

\draw[black!10, line width=5mm] (1,-0.3) -- (2,-0.3);
\node at (1.5,-0.3) {\textcolor{blue}{\footnotesize $\tg_{\ti_1,\ti_2}\approx \tw$}};
\draw[black!10, line width=5mm] (2,-0.3) -- (4,-0.3);
\node at (3,-0.3) {\textcolor{blue}{\footnotesize $\tg_{\ti_2,\ti_3}\approx \tw$}};
\draw[black!10, line width=5mm] (4,-0.3) -- (10,-0.3);
\node at (7,-0.3) {\textcolor{blue}{\footnotesize $\tg_{\ti_3,\ti_4}\approx \tw$}};
\draw[black!10, line width=5mm] (10,-0.3) -- (12,-0.3);

\fill[white] (0.97,-0.45) rectangle (1.03,-0.1);
\fill[white] (1.97,-0.45) rectangle (2.03,-0.1);
\fill[white] (3.97,-0.45) rectangle (4.03,-0.1);
\fill[white] (9.97,-0.45) rectangle (10.03,-0.1);

\draw[green!60!blue,densely dotted] (1,1) -- (1,-0.5); 
\draw[green!60!blue,densely dotted] (2,1/2) -- (2,-0.5); 
\draw[green!60!blue,densely dotted] (4,1/4) -- (4,-0.5); 
\draw[green!60!blue,densely dotted] (10,1/10) -- (10,-0.5); 

\node[draw=green!60!blue,fill=white] at (1,0.01) {\textcolor{black}{\footnotesize $\ti_1$}};
\node[draw=green!60!blue,fill=white] at (2,0.01) {\textcolor{black}{\footnotesize $\ti_2$}};
\node[draw=green!60!blue,fill=white] at (4,0.01) {\textcolor{black}{\footnotesize $\ti_3$}};
\node[draw=green!60!blue,fill=white] at (10,0.01) {\textcolor{black}{\footnotesize $\ti_4$}};

\end{axis}
\end{tikzpicture}
\caption{Time window and indices. Illustration for $\alpha_k = \frac1k$ and $\tw = 1$.}
    \label{fig:enter-label}
\end{figure}
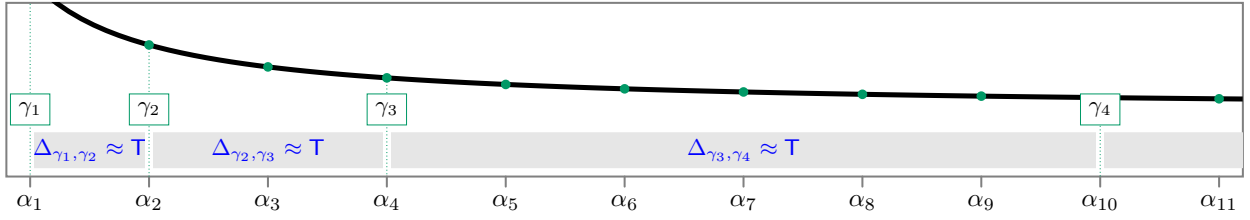

We now provide a connection between the step sizes $\{\alpha_k\}_k$ and the time window. 
\begin{lemma}\label{lem:time-length}
Assume that $\{\alpha_k\}_k$ satisfies $\lim_{k\to\infty} \alpha_k =0$ and $\sum_{k=1}^\infty \alpha_k = \infty$. Then, for any given time window $\tw>0$ and $\delta\in[0,1)$, there exists $K_\delta \in \N$ such that 
$\delta \tw \leq \tg_{\ti_k,\ti_{k+1}} \leq \tw$ for all $k \geq K_\delta$.
\end{lemma}

The proof of \Cref{lem:time-length} is elementary. Due to $\alpha_k \to 0$, there is $K\in\N$ such that $\alpha_k \leq (1-\delta)\tw \leq \tw$ for all $k\geq K$. This implies $\varpi(k,\tw) = \sup\{n \geq k: \tg_{k,n} \leq \tw\}$ for all $k \geq K$ and $\tg_{\ti_k,\ti_{k+1}} \leq \tw$ for all $k$ with $\ti_k \geq K$. By the optimality of $\ti_{k+1} = \varpi(\ti_k,\tw)$, we can further infer $ \tg_{\ti_k,\ti_{k+1}}  \geq \tw - \alpha_{\ti_{k+1}} \geq \delta \tw$ for all $k$ and $\ti_k$ satisfying $\ti_k \geq K$. Here, let us also refer to the proof of \cite[Lemma 6.2]{tadic2015convergence}.

The time window-based formalism will allow us to control certain aggregated error terms that are related to the stochastic errors $\{\se^k\}_k$. \\[1mm]
\noindent\textbf{Stochastic Error Estimates.} We consider step sizes $\{\alpha_k\}_k$ of the form 
\begin{equation}
    \label{step size-1}
  \{\alpha_k\}_k \subset \R_{++}\;\text{is non-increasing},\quad 
    {\sum}_{k=1}^\infty \alpha_k = \infty, \quad  {\sum}_{k=1}^\infty \alpha_k^2 \beta_k^2 <\infty,
\end{equation}
for some non-decreasing sequence $\{\beta_k\}_k \subset \R_{++}$. The condition \eqref{step size-1} can be satisfied by, for instance, polynomial step sizes  $\alpha_k \sim 1/k^\gamma$, $\gamma\in(\frac12,1]$, with $\beta_k \equiv 1$.

Based on the time window $\tw$ and the indices $\{\ti_k\}_k$, we introduce the aggregated error $\scs_k$ and the associated event $\cS$:
\begin{equation}
    \label{eq:def-stochastic err}
    \scs_k:=\max_{t\in\Ti_k} \Big\| {\sum}_{i=\ti_k}^{t-1} \alpha_i \se^i \Big\| \quad \text{and} \quad \cS := \left\{\omega \in \Omega: {\sum}_{k=1}^{\infty}\,\beta_{\ti_k}^2\scs_k^2(\omega) < \infty\right\},
\end{equation}
where $\Ti_k =\{t\in\N : \ti_k < t \leq \ti_{k+1}\}$. Next, we provide an almost sure bound for the aggregated error terms $\{\scs_k\}_k$ by showing that the event $\cS$ occurs with probability 1. The proof of \Cref{lem:err_estimate} is based on \cite[Lemma 6.1]{tadic2015convergence} and is presented in \Cref{proof:martingale}.
\begin{lemma}[Error estimate]
	\label{lem:err_estimate}
	Let \Cref{as:sgd} hold and suppose that $\{\alpha_k\}_k$ satisfies the condition \eqref{step size-1}. It then holds that $\Prob(\cS)=1$.
\end{lemma}

Note that \Cref{lem:err_estimate} provides an upper bound for the aggregated errors and implies $\scs_k = o(\beta_{\ti_k}^{-1})$ a.s.. Since $\{\beta_k\}_k \subseteq \R_{++}$ is assumed to be non-decreasing, we can further infer $\sum_{k=1}^\infty \scs_k^2 < \infty$ a.s..  

\subsection{Iterate Bounds and Descent-type Property} \label{subsec:merit func}

To study the convergence of $\SGDM$, we introduce the following auxiliary sequence $\{\sz^k\}_k$:
\begin{equation}
    \label{def:z}
    \sz^k := \frac{1}{1-\lambda} \sx^k - \frac{\lambda }{1-\lambda} \sx^{k-1}\quad \text{for all} \; k \geq 1.
\end{equation}
Auxiliary sequences and interpolations of this form are frequently used in the analysis of momentum methods, see, e.g., \cite{ghadimi2015global,liu2020improved}. Moreover, the definition \eqref{def:z} directly implies:
\begin{equation}
    \label{eq:lem-iter-bound-z-x}
    \sz^k-\sx^k = \frac{\lambda}{1-\lambda}(\sx^k-\sx^{k-1})\quad \text{and} \quad \sz^{k+1} = \sz^k - \frac{\alpha_k}{1-\lambda} \sg^k\quad \text{for all} \; k \geq 1. 
\end{equation}
We further define the stochastic process $\{\sd_k\}_k$:
\begin{equation}
    \label{def:d}
    \sd_k = \max\{{\max}_{\ell\in\Ti_k}\|\sx^\ell-\sx^{\ti_k}\|,\;{\max}_{\ell\in\Ti_k}\|\sz^\ell-\sz^{\ti_k}\|\}.
\end{equation}
Next, we provide several almost sure iterate bounds for terms involving the sequences $\{\sx^k\}_k$, $\{\sz^k\}_k$, and $\{\scs^k\}_k$. The proof of \Cref{lem:iterate-bound} is deferred to \cref{Proof:lem:iterate-bound}.

\begin{lemma}[Iterate bounds]\label{lem:iterate-bound}
    Let \Cref{as:sgd,Assumption:1} hold and let $\{\sx^k\}_k$ be generated by $\SGDM$ with $\lambda\in[0,1)$, $\mparam \geq 0$, and $\{\alpha_k\}_k$ satisfying \eqref{step size-1}. For every time window $\tw\in(0,\frac{(1-\lambda)^2}{20\sL(1+2\mparam)}]$ and its associated time indices $\{\ti_k\}_k$, there exists $K_\tw \geq 1$ such that the following bounds hold almost surely for all $k\geq K_\tw$:
    \begin{equation}
        \label{eq:tedious-01-later}
        \sd_k^2\leq  \frac{3}{2}\|\sz^{\ti_k} - \sx^{\ti_k}\|^2 + \frac{15}{(1-\lambda)^2}[\tw^2 \|\nabla f(\sz^{\ti_k})\|^2 + \scs_{k}^2],\quad \text{and}
    \end{equation}
    \begin{equation}
        \label{eq:tedious-02-later} \|\sz^{\ti_{k+1}}-\sx^{\ti_{k+1}}\|^2 \leq \frac{1+\lambda}{2}\|\sz^{\ti_k} - \sx^{\ti_k}\|^2 + \frac{8}{(1-\lambda)^3}[\tw^2\|\nabla f(\sz^{\ti_k})\|^2 + 4\scs_k^2].
    \end{equation}
\end{lemma}

In order to establish a descent-type property for $\SGDM$, we introduce a merit function $\mer:\Rn\times\Rn\to\R$. Setting $\zeta := \frac{3\sL}{1-\lambda}$, $\mer$ is defined via:
\begin{equation}
    \label{def:merit} 
    \mer(x,z):=f(z) + \zeta \|z-x\|^2 \quad \implies \quad
    \nabla \mer(x,z) = \begin{bmatrix}
    2\zeta(x-z)\\\nabla f(z) + 2\zeta(z-x)
    \end{bmatrix}.
\end{equation}

Similar merit functions --- typically of the form ``$f(x^k)+\zeta\|x^k-x^{k-1}\|^2$'' --- have been employed in the analysis of momentum-type methods \cite{ruszczynski1987linearization,zavriev1993heavy,OchCheBroPoc14,josz2023convergence}. In our setting, analogous usage of ``$f(x^k)$'' in $\mer$ causes additional complications. Thus, our merit function is based on ``$f(z^k)$'' instead of ``$f(x^k)$''. Below, we list several important bounds involving $\nabla \mer(x,z)$:
\begin{equation} \label{grad-ineq}
    \begin{aligned}
    & 4\zeta^2\|x-z\|^2 \leq \|\nabla \mer(x,z)\|^2, \quad \|\nabla f(z)\|^2 \leq 2\|\nabla \mer(x,z)\|^2, \\
    & \|\nabla \mer(x,z)\|^2 \leq 12 \zeta^2 \|z-x\|^2 + 2 \|\nabla f(z)\|^2.
\end{aligned}
\end{equation}
The latter inequalities follow from
\[
\begin{aligned}
\|\nabla \mer(x,z)\|^2 &= 4\zeta^2\|x-z\|^2 + \|\nabla f(z) + 2\zeta(z-x)\|^2 \\&= 8\zeta^2\|x-z\|^2 + \|\nabla f(z)\|^2 +  4\zeta \langle \nabla f(z), z-x \rangle,
\end{aligned}
\]
and $-\frac{\varepsilon}{2}\|\nabla f(z)\|^2 - \frac{8\zeta^2}{\varepsilon}\|z-x\|^2 \leq 4\zeta\langle \nabla f(z), z-x \rangle \leq \|\nabla f(z)\|^2 + 4\zeta^2\|z-x\|^2$ for $\varepsilon \in \{1,2\}$.
We now present an approximate, time window-based descent property for $\SGDM$. The proof can be found in \Cref{proof:lem:approx-desent}.

\begin{lemma}[Approximate descent]\label{lem:approx-desent}
     Suppose \Cref{as:sgd,Assumption:1} hold and $\SGDM$ is run with $\lambda\in[0,1)$, $\nu\geq0$, and $\{\alpha_k\}_k$ fulfilling \eqref{step size-1}. For any time window $\tw\in(0,\frac{(1-\lambda)^3}{50\sL(1+2\mparam)^2}]$ and its associated time indices $\{\ti_k\}_k$, there is $K_\tw \geq 1$ such that 
\[
    \mer(\sx^{\ti_{k+1}},\sz^{\ti_{k+1}}) + \frac{\sL}{12}\sd_k^2  + \frac{\tw\|\nabla \mer(\sx^{\ti_{k}},\sz^{\ti_{k}})\|^2}{100(1-\lambda)}
    \leq \mer(\sx^{\ti_{k}},\sz^{\ti_{k}})  + \frac{8\scs_k^2}{(1-\lambda)\tw}.
\]
a.s. for all $k \geq K_\tw$.
\end{lemma}

\section{Convergence Analysis}\label{sec:convergence analysis} 
In this section, we provide our main convergence results for $\SGDM$ including global convergence and iterate convergence under the {\L}ojasiewicz inequality. Throughout this section, we utilize the fixed time window 
\[\tw := \frac{(1-\lambda)^3}{50\sL(1+2\mparam)^2}\quad \text{and}\quad \text{$\{\ti_k\}_k$ denotes the associated time indices.}\]
Such choice of $\tw$ allows us to apply the results in \Cref{lem:err_estimate,lem:iterate-bound,lem:approx-desent}.
\subsection{Global Convergence}\label{sec:subsequential}
We first present a basic global convergence result.
\begin{theorem}
	\label{Prop:convergence}
	Let \Cref{as:sgd,Assumption:1} hold and let $\{\sx^k\}_k$ be generated by $\SGDM$ with $\lambda\in[0,1)$, $\mparam\geq0$, and $\{\alpha_k\}_k$ satisfying
	\begin{equation}\label{eq:ass-step}
	\{\alpha_k\}_k \subset \R_{++}\;\text{is non-increasing}, \quad {\sum}_{k=1}^\infty \alpha_k = \infty, \quad {\sum}_{k=1}^\infty \alpha_k^2 < \infty. 
	\end{equation}
	Then, the following statements are valid:
\begin{enumerate} [label=\textup{(\alph*)},topsep=0pt,itemsep=0ex,partopsep=0ex]
     \item It holds that $\lim_{k\to\infty}\sd_k=0$, $\lim_{k\to\infty}\|\sx^k-\sz^k\|=0$, $\lim_{k\to\infty} \|\nabla{f}(\sx^k)\|=0$, and $\lim_{k\to\infty} \|\nabla{f}(\sz^k)\|=0$ a.s..
     \item $\{f(\sx^k)\}_{k}$ and $\{f(\sz^k)\}_k$ converge to some $\sct{f}^* :\Omega \to \R$ a.s..
 \end{enumerate} 
\end{theorem}

\begin{proof}{Proof}
Set $\beta_k = 1$ in \eqref{step size-1} and \eqref{eq:def-stochastic err}. By \Cref{lem:err_estimate}, we have $\Prob(\cS)=1$, where $\cS=\{\omega\in\Omega:{\sum}_{k=1}^\infty \;\scs_k^2(\omega) < \infty\}$. We fix an arbitrary sample $\omega \in \cS$ and consider $x^k \equiv \sx^k(\omega)$, $z^k \equiv \sz^k(\omega)$, $s_k \equiv \scs_k(\omega)$, $d_k\equiv \sd_k(\omega)$, etc.. By \Cref{lem:approx-desent}, there is $K_\tw \geq 1$ such that for all $k\geq K_\tw$, 
\begin{equation}\label{eq:approximate descent use}
  \mer(x^{\ti_{k+1}},z^{\ti_{k+1}}) + u_{k+1} + \frac{\sL}{12}d_k^2 + \frac{\tw\|\nabla \mer(x^{\ti_k}, z^{\ti_k})\|^2}{100(1-\lambda)}  \leq \mer(x^{\ti_{k}},z^{\ti_{k}}) + u_{k},
\end{equation}
where $u_k:=\frac{8}{(1-\lambda)\tw}{\sum}_{i=k}^\infty s_i^2$. Since $f$ is bounded from below, we have $\mer(x^{\ti_{k}},z^{\ti_{k}}) + u_{k}\downarrow f^*$ for some $f^*\in\R$. Telescoping \eqref{eq:approximate descent use} and using \cref{grad-ineq}, this yields
  \begin{equation}
     \label{eq:prop:convergence-important}
     \|\nabla f(z^{\ti_{k}})\| \to 0, \quad d_k\to0, 
 \quad \text{and} \quad {\max}_{\ell\in\Ti_k}\|z^{\ell}-x^{\ell}\| \to 0\quad \text{as\quad $k\to\infty$}.
\end{equation}
Here, the last result follows from \eqref{eq:lem-iter-bound-z-x} and the subsequent relations: 
\[
    (1-\lambda) \|z^{m}-x^{m}\|={\lambda}\|x^m-x^{m-1}\| \leq {\lambda}[\|x^m-x^{\ti_k}\| + \|x^{m-1}-x^{\ti_k}\|]  \leq {2\lambda}d_k\to0, \quad \forall~m\in\Ti_k. 
\]
Based on \eqref{eq:prop:convergence-important} and applying the Lipschitz continuity of $\nabla f$, we obtain
\[
 \begin{aligned}
   {\max}_{\ell\in\Ti_k} \|\nabla f(x^\ell)\| &\leq \|\nabla f(x^{\ti_{k}})\| +   {\max}_{\ell\in\Ti_k} \sL\|x^\ell-x^{\ti_k}\| \leq \|\nabla f(x^{\ti_{k}})\| +   \sL d_k \\&\leq \|\nabla f(z^{\ti_{k}})\| +\sL \|z^{\ti_k}-x^{\ti_k}\| +  \sL d_k \to 0 \quad \text{as\quad $k\to\infty$.}
 \end{aligned}
 \]
Therefore, we can infer $\|\nabla f(x^k)\|, \|\nabla f(z^k)\| \to 0$ as $\|x^k-z^k\|\to0$.

Since $\{\mer(x^{\ti_{k}},z^{\ti_{k}}) + u_{k}\}_k$ converges to $f^*$ and it holds that $u_k, \|z^{\ti_k}-x^{\ti_k}\|\to0 $ as $k\to\infty$, we conclude that $\lim_{k\to\infty} f(z^{\ti_k}) = f^*$. By \Cref{Assumption:1}, we have
\begin{equation}
    \label{eq:prop-descent-use-later}
    \begin{aligned}
        &\hspace{-1ex}|f(y_1) - f(y_2)| \leq  \max\{\|\nabla f(y_1)\|,\|\nabla f(y_2)\|\}\cdot \|y_1-y_2\| + \tfrac{\sL}{2}\|y_1-y_2\|^2 \\ &\leq \tfrac{1}{2\sL}\max\{\|\nabla f(y_1)\|^2,\|\nabla f(y_2)\|^2\} + \sL\|y_1-y_2\|^2,\quad \text{for all}\;\; y_1,y_2\in\Rn.
    \end{aligned}
\end{equation}
Substituting $y_1=x^{\ti_k}$, $y_2=z^{\ti_k}$ in \eqref{eq:prop-descent-use-later} and using $\|z^{\ti_k}-x^{\ti_k}\| \to 0$, we obtain \[|f(x^{\ti_k})-f(z^{\ti_k})| \leq \tfrac{1}{2\sL}\max\{\|\nabla f(x^{\ti_k})\|^2,\|\nabla f(z^{\ti_k})\|^2\} + \sL\|x^{\ti_k}-z^{\ti_k}\|^2\to0.\]
Furthermore, substituting $y_1=x^{\ti_k}$, $y_2=x^{\ell}$ in \eqref{eq:prop-descent-use-later}, it follows 
\[
\max_{\ell\in\Ti_k} |f(x^\ell)-f(x^{\ti_k})| \leq \frac{1}{2\sL}\max_{\ell\in\Ti_k\cup\{\gamma_k\}} \|\nabla f(x^\ell)\|^2 + \sL \cdot \max_{\ell\in\Ti_k} \|x^\ell-x^{\ti_k}\|^2\to0, \;\; k\to\infty.
\]
This implies $\max_{\ell\in\Ti_k} |f(x^\ell)-f^*| \leq |f(x^{\ti_k}) - f^*| + \max_{\ell\in\Ti_k} |f(x^\ell)-f(x^{\ti_k})| \to 0 $ and thus, we have $f(x^k) \to f^*$ as $k\to\infty$. Noting $\|x^k-z^k\|\to0$, we also conclude that $\lim_{k\to\infty} f(z^k) = f^*$.
\end{proof}
	
\subsection{Iterate Convergence under the {\L}ojasiewicz  Inequality} \label{sec:sequential}
In this subsection, we establish iterate convergence results, i.e., the stochastic process $\{\sx^k\}_{k}$ generated by $\SGDM$ is shown to converge to a $\crit(f)$-valued mapping $\sx^*:\Omega\to\crit(f)$ almost surely. This type of convergence is typically interpreted as a \emph{last-iterate convergence} result, previously studied in the (strongly) convex setting, \cite{gadat2018stochastic,sebbouh2021almost}.

\begin{theorem}
	\label{thm:finite sum}
	 Suppose \Cref{as:sgd,Assumption:1,Assumption:2} hold and let $\{\sx^k\}_k$ be generated by $\SGDM$ with $\lambda \in [0,1),\; \mparam \geq 0$, and non-increasing step sizes $\{\alpha_k\}_k\subset \R_{++}$ satisfying: 
	\begin{equation}\label{eq:kl-step}  
		{\sum}_{k=1}^\infty \alpha_k = \infty \quad \text{and} \quad {\sum}_{k=1}^{\infty} \alpha_k^2 \Big( {\sum}_{i=1}^k\alpha_i \Big)^{2r}< \infty \quad \text{for some $r>\frac12$}.
	\end{equation}
 Then, the event 
$\mathcal I := \{\omega\in\Omega: {\lim}_{k\to\infty}\,\|\sx^k(\omega)\|=\infty \; \text{or}\; \sx^k(\omega) \to x^*\in\crit(f) \}$ occurs w.p. $1$.
 \end{theorem}
 
 \begin{remark} \label{rem:iter conv}
 We continue with several remarks on \Cref{thm:finite sum}.
 \begin{itemize}[leftmargin=4.5ex]
 \item \emph{No a priori boundedness.} The statement $\Prob(\mathcal I) = 1$ can be interpreted as follows: If the stochastic process $\{\sx^k\}_k$, generated by $\SGDM$, does not diverge to infinity a.s., then it must converge to some stationary points a.s.. The case $\|\sx^k(\omega)\| \to \infty$ can be ruled out if $\{\sx^k(\omega)\}_k$ has at least one accumulation point. Alternatively, we may define the event 
\[
\cX:=\{\omega\in\Omega: {\liminf}_{k\to\infty} \; \|\sx^k(\omega)\|<\infty \}.
\]
The iterates $\{\sx^k\}_k$ then converge a.s$.$ on $\cX$. Moreover, if $\Prob(\cX)=1$, then $\{\sx^k\}_k$ converges a.s$.$ to some stationary points of $f$. The condition $\Prob(\cX)=1$ is significantly weaker than the a priori (a.s.) boundedness of $\{\sx^k\}_k$ required in other {\L}ojasiewicz inequality-based convergence analyses of stochastic methods, \cite{tadic2015convergence,li2023convergence}. 
 \item \emph{Step sizes requirements.} Condition \eqref{eq:kl-step} appeared in a similar form in \cite[Corollary 2.2]{tadic2015convergence} requiring $r > 1$. \eqref{eq:kl-step} holds, e.g., for polynomial step sizes of the form $\alpha_k \sim k^{-\gamma}$, $\gamma\in(\frac23,1]$. Hence, under proper polynomial step sizes, iterate convergence of $\{\sx^k\}_k$ is ensured by  \Cref{thm:finite sum}. 
 \end{itemize}
 \end{remark} 
 
\subsection{Proof of Iterate Convergence Theorem} \label{proof:thm:finite sum}
Before proving \Cref{thm:finite sum}, we first state two preparatory lemmas.
\begin{lemma}\label{lem:mer-kl}
    If $f:\Rn\to\R$ satisfies the {\L}ojasiewicz inequality at $x^*\in\Rn$ with exponent $\theta\in[\frac12,1)$, then the merit function $\mer:\Rn\times \Rn \to \R$ has the {\L}ojasiewicz inequality at $(x^*,x^*)$ with exponent $\theta\in[\frac12,1)$, i.e.,  there are $\eta \in (0,\infty]$ and a neighborhood $U(x^*)$ of $x^*$ such that for all $x,z \in V(x^*) := U(x^*) \cap \{x \in \Rn: |f(x) - f(x^*)| < \eta\}$, we have
  \[\|\nabla \mer(x,z)\|\geq \sC|\mer(x,z)-f(x^*)|^\theta, \quad \text{for some } \sC = \sC(x^*) >0. 
  \]
\end{lemma}

\begin{proof}{Proof} This result essentially follows from the proof of \cite[Theorem 3.6]{li2018calculus}. Note that the condition ``$|f(x)-f(x^*)| > 0$'' can be dropped here, as the {\L}ojasiewicz inequality in \Cref{Assumption:2} and \Cref{lem:mer-kl} trivially hold in the case $f(x) = f(x^*)$. We will omit further details and refer to \cite{li2018calculus}. \end{proof} 

Leveraging \Cref{lem:mer-kl}, we can establish the following technical trajectory-based bound. The proof of \Cref{lem:kl-ineq} is presented in \Cref{proof:lem:kl-ineq}.
\begin{lemma}\label{lem:kl-ineq}
    Let \Cref{as:sgd,Assumption:1,Assumption:2} hold and let $\{\sx^k\}_k$ be generated by $\SGDM$ with $\lambda \in [0,1)$, $\mparam \geq 0$, and $\{\alpha_k\}_k$ satisfying \eqref{step size-1}. Consider fixed $\omega\in\cS$, $x^k\equiv\sx^k(\omega)$, $z^k\equiv\sz^k(\omega)$, $d_k\equiv \sd_k(\omega)$, $s_k\equiv \scs_k(\omega)$, etc.. If there are $x^*\in\crit(f)$ and $k\geq K_\tw$ such that $x^{\ti_k},z^{\ti_k} \in V(x^*)$, $|\mer(x^{\ti_k},z^{\ti_k}) - f(x^*)| < 1$, and $\mer(x^{\ti_{j}},z^{\ti_{j}}) + u_j \geq f(x^*)$, for $j \in \{k,k+1\}$, then it holds that
    \[ d_k \leq \frac{150(1+2\mparam)}{(1-\lambda)^3}[\Psi_k - \Psi_{k+1}] + 3(1+2\mparam)\sC\tw u_k^\vartheta,\quad \text{for all $\vartheta\in[\theta,1)$},
 \]
 where $u_k:=\frac{8}{(1-\lambda)\tw}{\sum}_{i=k}^\infty s_i^2$ and $\Psi_k := \frac{1}{\sC(1-\vartheta)}[\mer(x^{\ti_{k}},z^{\ti_{k}}) - f(x^*) + u_k]^{1-\vartheta}$. 
\end{lemma}

\emph{Proof of \Cref{thm:finite sum}.\ }
Setting $\beta_k=( \sum_{i=1}^k\alpha_i )^{r}$ in \eqref{step size-1} and noting $\sum_{k=1}^{\infty} \alpha_k=\infty$, we have $\Prob(\cS)=1$ according to \Cref{lem:err_estimate}. Moreover, the condition $\sum_{k=1}^{\infty} \alpha_k^2 \beta_k^{2}< \infty$ readily implies $\sum_{k=1}^{\infty} \alpha_k^2<\infty$. Thus, \Cref{Prop:convergence} is applicable. 
Let us define the master event $\cE$:
\begin{equation}
    \label{def:Event}
    \begin{aligned}
        \cE:=\cS& \cap \{\omega\in\Omega:\exists\; f^*\in\R\;\;   \text{s.t.}\;\; f(\sx^k(\omega)) \to f^*\; \text{and}\; f(\sz^k(\omega)) \to f^*\} \\&\cap \{\omega\in\Omega:\nabla f(\sx^k(\omega)) \to0,\, \sd_k(\omega)\to0\;\text{and}\;\|\sx^k(\omega)-\sz^k(\omega)\|\to0\}.
    \end{aligned}
\end{equation}
Clearly, $\Prob(\cE)=1$ due to \Cref{lem:err_estimate} and \Cref{Prop:convergence}. 
Let us fix $\omega\in\cE$ and consider the realizations $x^k\equiv\sx^k(\omega)$, $z^k\equiv\sz^k(\omega)$, $d_k\equiv \sd_k(\omega)$, $s_k\equiv \scs_k(\omega)$, etc..

If $\|x^k\| \nrightarrow \infty$, then $\{x^k\}_k$ has at least one accumulation point $x^*\in\Rn$. Since $\nabla f(x^k)\to0$, we conclude that $x^*\in\crit(f)$ and there exists a subsequence $\{x^{\ell_k}\}_k \subseteq \{x^k\}_k$ converging to $x^*$. Since $\{f(x^k)\}_k$ and $\{f(z^k)\}_k$ converge to some $f^*\in\R$ and we have $f(x^{\ell_k})\to f(x^*)$, 
we can infer $f(x^*)=f^*$. Thus, there is $K_f \geq 1$ such that 
\begin{equation}
    \label{eq:thm:kl-1}
  \max\{ |f(x^k) - f(x^*)|,|f(z^k) - f(x^*)|\} < \min\{1,\eta\}\quad \text{for all}\;k \geq K_f.
\end{equation}

Next, we verify that $\{u_k^\vartheta\}_k$ is summable. Due to $\sum_{k=1}^\infty \beta_{\ti_k}^2 s_k^2 < \infty$ and $\beta_k\to\infty$, there is $K \geq  \max\{K_\delta,K_f\}$ such that $\sum_{k=t}^\infty \beta_{\ti_k}^2 s_k^2 <1$, $\beta_{\ti_t} > 1$, and $\tg_{\gamma_t,\gamma_{t+1}} \geq \delta\tw$ for all $t \geq K$ ($K_\delta$ is defined as in \Cref{lem:time-length}). Since $\{\beta_k\}_k$ is non-decreasing and we have $\vartheta r> \frac12$, it follows
\begin{equation*}
    \label{thm:sum-s_k}
    {\sum}_{k=t}^\infty \Big({\sum}_{i=k}^\infty s_i^2\Big)^\vartheta \leq  {\sum}_{k=t}^\infty \Big(\beta_{\ti_k}^{-2}{\sum}_{i=k}^\infty \beta_{\ti_i}^2s_i^2\Big)^{\vartheta} \leq
    {\sum}_{k=t}^\infty \Big( {\sum}_{i=1}^{\ti_k}\alpha_i \Big)^{-2\vartheta r} <\infty,
\end{equation*}
where the last inequality uses ${\sum}_{i=1}^{\ti_k}\alpha_i \geq \sum_{i=K}^{k-1} \tg_{\gamma_i,\gamma_{i+1}} + {\sum}_{i=1}^{\ti_K-1}\alpha_i \geq \delta \tw(k-K)$. Hence, by the definition of $\{u_k\}_k$, we conclude that
\begin{equation}
    \label{thm:sum-u_k}
{\sum}_{k=t}^\infty u_k^\vartheta \to 0 \quad \text{as $t$ tends to infinity.}   
\end{equation}

Due to $x^{\ell_k} \to x^*$, $\max_{\ell\in\Ti_k}\|x^\ell - x^{\ti_k}\|\to0$, and $\bigcup_k \Gamma_k = \N$, there is a subsequence of $\{x^{\ti_k}\}_k$ converging to $x^*$. Moreover, utilizing \Cref{lem:approx-desent}, $f(z^k) \to f(x^*)$, $\|x^k-z^k\|\to0$, and $u_k\to0$, it holds that $\mer(x^{\ti_k}, z^{\ti_k}) + u_k \downarrow f(x^*)$ and $\Psi_k\downarrow0$. Hence, for any given $\rho>0$ with $\cB_{\rho}(x^*) \subseteq U(x^*)$,  there is $t \geq K$ such that $|\mer(x^{\ti_t},z^{\ti_t}) - f(x^*)| < 1$, $z^{\ti_k}, x^{\ti_k} \in V(x^*)$ (cf$.$ \eqref{eq:thm:kl-1}), and
\begin{equation}
	\label{eq:thm-kl-rho}
\max\{\|x^{\ti_{t}} - x^*\|,\|z^{\ti_{t}} - x^*\|\}  + \sM \Psi_{t}+ 3(1+2\mparam)\sC\tw {\sum}_{i=t}^\infty u_i^\vartheta < \rho, 
\end{equation}
where $\sM := \frac{150(1+2\mparam)}{(1-\lambda)^3}$. We now show that the following statements hold for all $k \geq t$: 
	\begin{enumerate}[label=(\alph*),topsep=4pt,itemsep=0ex,partopsep=0ex]
	\item $x^{\ti_k},z^{\ti_k} \in \mathcal{B}_{\rho}(x^*)$ and $|f(x^{\ti_k}) - f^*| < \min\{1,\eta\}$, $|f(z^{\ti_k}) - f^*| < \min\{1,\eta\}$.
 \vspace{1mm}
	\item $\sum_{i=t}^{k}d_i  \leq \sM[\Psi_t - \Psi_{k+1}] + 3(1+2\mparam)\sC\tw \sum_{i=t}^k u_i^\vartheta.$
	\end{enumerate}
We prove these statements by induction. By \Cref{lem:kl-ineq} and the definition of the index $t$, (a) and (b) hold for $k=t$. Assume there is $m > t$ such that the statements (a) and (b) are valid for $k=m$. Let us consider $k=m+1$. From \eqref{eq:thm:kl-1}, it follows $\max\{|f(x^{\ti_{m+1}}) - f^*|,|f(z^{\ti_{m+1}}) - f^*|\} < \min\{1,\eta\}$. Next, we show that $x^{\ti_{m+1}},z^{\ti_{m+1}}\in \mathcal{B}_\rho(x^*)$. Using the triangle inequality and statement (b) for $k=m$, 
	\[\begin{aligned}
		\|x^{\ti_{m+1}} - x^*\| &\leq \|x^{\ti_{m+1}} - x^{\ti_{m}}\| + \|x^{\ti_{m}} - x^{\ti_{t}} \| +\|x^{\ti_{t}} - x^*\| \leq \|x^{\ti_{t}} - x^*\| +  {\sum}_{i=t}^{m}d_i \\ &\leq \|x^{\ti_{t}} - x^*\| + \sM[\Psi_{t} - \Psi_{m+1}] + 3(1+2\mparam)\sC\tw {\sum}_{i=t}^m u_i^\vartheta < \rho,
	\end{aligned}\]
	where the last inequality follows from \eqref{eq:thm-kl-rho} and $\Psi_{k} \geq 0$ for all $k\geq t$. Repeating this step for $z^{\ti_{m+1}}$, we can also show $z^{\ti_{m+1}}\in \mathcal{B}_\rho(x^*)$. This proves (a) for $k={m+1}$, implying that $x^{\ti_{m+1}}, z^{\ti_{m+1}} \in U(x^*)$ and $\max\{|f(x^{\ti_{m+1}}) - f^*|,|f(z^{\ti_{m+1}}) - f^*|\} < \min\{1,\eta\}$. Hence, \Cref{lem:kl-ineq} is applicable for $k=m+1$, i.e., we have
	\[d_{m+1} \leq \sM[\Psi_{m+1} - \Psi_{m+2}] + 3(1+2\mparam)\sC\tw u_{m+1}^\vartheta.\]
	Combining this inequality with the bound in (b) for $k=m$, we can infer ${\sum}_{i=t}^{m+1}d_i  \leq \sM[\Psi_t - \Psi_{m+2}] + 3(1+2\mparam)\sC\tw {\sum}_{i=t}^{m+1} u_i^\vartheta$, i.e.,
	 (b) is also valid for $k=m+1$. 
	 Therefore, both (a) and (b) hold for all $k \geq t$. Statement (b) and \eqref{eq:thm-kl-rho} then imply
	 \[
	 {\sum}_{k=t}^{\infty}d_k \leq \sM\Psi_t  + 3(1+2\mparam)\sC\tw {\sum}_{i=t}^\infty u_i^\vartheta < \rho <\infty.
	 \]
  According to the definition \eqref{def:d} of $d_k$ and following the discussions in \Cref{subsec:toy-example}, this summability condition ensures that $\{x^k\}_k$ is a Cauchy sequence and hence, the whole sequence $\{x^k\}_k$ converges to the stationary point $x^*$. Recalling $\Prob(\cE)=1$, this result holds in an almost sure sense. \hfill $\square$

\section{Convergence Rates}
As discussed in \Cref{rem:iter conv}, the convergence result in \Cref{thm:finite sum} guarantees that the sequence $\{\sx^k\}_k$ converges a.s$.$ to some $\sx^*:\Omega\to\crit(f)$ on $
\cX=\{\omega\in\Omega: {\liminf}_{k\to\infty} \; \|\sx^k(\omega)\|<\infty \}.$  Since the goal of this section is to quantify the (local) rates of $\SGDM$, we will investigate the behavior of $\{\sx^k\}_k$ on $\cX$. If $\Prob(\cX)=1$, then all of our following results hold almost surely.

\subsection{Main Result and Discussions}
We now provide convergence rates for $\SGDM$.
\begin{theorem}
\label{thm:convergence-rate}
Let \Cref{as:sgd,Assumption:1,Assumption:2} hold and let the sequence $\{\sx^k\}_k$ be generated by $\SGDM$ with $\lambda \in [0,1)$, $\nu \geq 0$. For a given function $g: \R \to \R_{++}$, we consider non-increasing step sizes $\{\alpha_k\}_k$ satisfying
\begin{equation} \label{eq:kl-stepsizes} 
{\sum}_{k=1}^{\infty}\alpha_k=\infty \quad \text{and} \quad {\sum}_{k=1}^{\infty}\alpha_k^2 g(\tg_k)^{2} <\infty, \quad \text{where} \quad \tg_k := {\sum}_{i=1}^{k}\alpha_i.
\end{equation}
\begin{enumerate}[label=\textup{\textrm{(\alph*)}},leftmargin=2em,topsep=0pt,itemsep=0ex,partopsep=0ex]
    \item If $g(x):=x^r$, $r>\frac12$, then $\{\sx^k\}_k$ converges to a $\crit(f)$-valued mapping $\sx^*:\Omega\to\crit(f)$ a.s$.$ on $\cX$. Moreover, there is an event $\cR$ with $\Prob(\cX) = \Prob(\cR)$ such that for all $\omega\in\cR$, the corresponding realizations $x^k \equiv \sx^k(\omega)$, $x^* \equiv \sx^*(\omega)$ satisfy
    \[
    \|x^k - x^*\| = \cO(\tg_k^{-\phi(\theta)}) \;\; \text{and} \;\;  \max\{|f(x^k)- f(x^*)|,\|\nabla f(x^k)\|^2\}= \cO(\tg_k^{-\psi(\theta)}),
    \]
     where $\theta\in[\frac12,1)$ denotes the {\L}ojasiewicz exponent at $x^*$ and the rate mappings $\phi, \psi$ are defined as $\phi(\theta):=\min\{r-\tfrac12,\tfrac{1-\theta}{2\theta-1}\}$ and $\psi(\theta):= \min\{2r,  \tfrac{1}{2\theta-1}\}$.
 \item If $g(x):=x^p \cdot \exp(rx)$, $r>0$, $p \geq 0$, then $\{\sx^k\}_k$ converges to a $\crit(f)$-valued mapping $\sx^*:\Omega\to\crit(f)$ a.s$.$ on $\cX$.  Moreover, there is $\cR\in\cF$ with $\Prob(\cX) = \Prob(\cR)$ such that for all $\omega\in\cR$, the realizations $x^k \equiv \sx^k(\omega)$, $x^* \equiv \sx^*(\omega)$ satisfy
      \[
    \|x^k - x^*\| = \cO({g(\tg_k)^{-1}}) \;\; \text{and} \;\;  \max\{|f(x^k)- f(x^*)|,\|\nabla f(x^k)\|^2\}= \cO({g(\tg_k)^{-2}}),
    \]
    when $\theta=\frac12$ and $r<\sC^2/400$. Here, the constant $\sC$ is from \Cref{lem:mer-kl}.
\end{enumerate}
\end{theorem}

The derivation of rates for $\SGDM$ is technical and is deferred to \Cref{proof:thm:convergence-rate}. Here, we would like to highlight how our proof strategy differs from other existing ones. In the classical deterministic setting, {\L}ojasiewicz inequality-based rate analyses typically necessitate the step sizes $\{\alpha_k\}_k$ to be constant or bounded away from zero, see, e.g., \cite{AttBol09,BolSabTeb14,fragarpey15,li2018calculus,OchCheBroPoc14}. By contrast, in $\SGDM$, we need to utilize diminishing step sizes, $\alpha_k \to 0$, to mitigate the stochastic errors and to ensure convergence. Hence, existing techniques for deterministic algorithms can not be directly applied. A more related rate analysis with diminishing step sizes appeared earlier in \cite{li2023convergence} for the random reshuffling ($\RR$) method. However, the update errors in $\RR$ (when approximating full gradients) can be controlled in a fully deterministic way and hence, the strategies in \cite{li2023convergence} can not be transferred to $\SGDM$. Additional comparisons are made after \Cref{coro:local-rate}. The verification of \Cref{thm:convergence-rate} will again rely on the introduced time window techniques and on suitable aggregations of the stochastic errors. 

The convergence rates in \Cref{thm:convergence-rate} are applicable to a wide spectrum of step size strategies. To use \Cref{thm:convergence-rate}, we simply need to verify if the step sizes fulfill the conditions in \eqref{eq:kl-stepsizes}. Substituting explicit expressions for $\tg_k$ in \Cref{thm:convergence-rate} then yields specific convergence rates. In the following, we illustrate the application of \Cref{thm:convergence-rate} for polynomial step sizes. The proof is presented in \Cref{proof:coro rate}.

\begin{corollary}
	\label{coro:local-rate}
	Under \Cref{as:sgd,Assumption:1,Assumption:2}, let $\{\sx^k\}_k$ be generated with $\alpha_k = \alpha/(k+\beta)^\gamma$ and $ \alpha>0$, $\beta,\nu\geq0$, $\gamma\in(\tfrac23,1]$, and $\lambda \in [0,1)$. Then, $\{\sx^k\}_k$ converges a.s$.$ to some $\sx^*:\Omega\to\crit(f)$ on $\cX$. Moreover, there is $\cR\in\cF$ with $\Prob(\cX) = \Prob(\cR)$ such that for all $\omega\in\cR$, the following rates hold for realizations $x^k \equiv \sx^k(\omega)$, $x^* \equiv \sx^*(\omega)$:
\begin{enumerate}[label=\textup{\textrm{(\alph*)}},leftmargin=2em,topsep=2pt,itemsep=0.5ex,partopsep=0ex]
		\item If $\gamma \in (\frac23,1)$, then for any $\varepsilon>0$, we have $\|x^k - x^*\| = o(1/k^{\Phi(\gamma,\theta) - \varepsilon})$,
  \[|f(x^k)- f(x^*)|= o(1/k^{\Psi(\gamma,\theta) - \varepsilon}) \;\; \text{and} \;\;  \|\nabla f(x^k)\|^2 = o(1/k^{\Psi(\gamma,\theta) - \varepsilon}),\]
  where $\theta\in[\frac12,1)$ is the {\L}ojasiewicz exponent at $x^*$ and the mappings $\Phi,\Psi$ are defined as $\Phi(\gamma,\theta):=\min\{\tfrac32\gamma-1,\tfrac{(1-\gamma)(1-\theta)}{2\theta-1}\}$ and $\Psi(\gamma,\theta):=\min\{2\gamma-1,\tfrac{1-\gamma}{2\theta-1}\}$.
  \item If $\gamma=1$, $\theta=\frac12$ and $\alpha>200/\sC^2$, it holds that $\|x^k - x^*\| = o(\log(k)^{\tfrac12+\varepsilon}/\sqrt{k})$,
  \[|f(x^k)- f(x^*)|= o(\log(k)^{1+\varepsilon}/k) \;\; \text{and} \;\;  \|\nabla f(x^k)\|^2 = o(\log(k)^{1+\varepsilon}/k), \quad \forall \; \varepsilon>0,\]
 where $\theta$ and $\sC$ denote the {\L}ojasiewicz exponent and constant (of $f$ and $\mer$) at $x^*$.
  \end{enumerate}
\end{corollary}

\Cref{thm:convergence-rate} and \Cref{coro:local-rate} provide novel insights into the asymptotic behavior of $\SGDM$. Specifically, the derived rates are obtained under a more general framework while being faster and improving the existing results even for $\SGD$. We now provide further comparisons and remarks on a.s$.$ convergence rates. 
\begin{itemize}[leftmargin=4.5ex]
    \item \emph{Function value rates.} Liu and Yuan \cite[Theorems {8 and 13}]{liu2022almost} establish convergence rates for $\SHB$ in the strongly convex and convex case\footnote{{Liu and Yuan \cite[Theorems {9 and 13}]{liu2022almost} also establish similar rates for SNAG in the same strongly convex / convex setting.}}. Using polynomial step sizes, $\alpha_k \sim 1/k^{\gamma}$, (a.s.) rates of the form $f(x^k) - f(x^*) = o(1/k^{1-\varepsilon})$ for strongly convex $f$ and $f(x^k) - f(x^*) = o(1/k^{\frac13 -\varepsilon})$ for general convex $f$ are shown in \cite{liu2022almost}. Our results in \Cref{coro:local-rate} can recover the rates in the strongly convex case (i.e., when $\theta = \frac12$) and we can achieve the improved rate $o(\log(k)^{1+\varepsilon}/k)$ when $\gamma=1$. Due to $\Psi(\gamma,\theta) > \frac13$ for all $\theta\in [\frac12,1)$, we generally obtain faster convergence compared to the convex setting studied in \cite{liu2022almost}. For $\SHB$ with particular parameter choices (i.e., $\lambda = \lambda_k\to1$ and $\alpha_k\sim 1/k^{\frac32+\varepsilon}$), Sebbouh et al$.$ derived the rate $f(x^k) - f(x^*) = o(1/k^{\frac12-\varepsilon})$ in the convex case, \cite[Corollary 17]{sebbouh2021almost}. Based on \Cref{coro:local-rate}, in the nonconvex setting, we can match this result if $\gamma \geq \frac34$ and $\gamma+\theta \leq \frac32$.
\item \emph{Iterate rates.} In \cite[Theorem 2.2]{tadic2015convergence}, Tadi\'{c} established rates for $\SGD$-iterates that are more related to our results. If $\gamma\in(\frac34,1)$, Tadi\'{c}'s rates are given by 
\[ \|x^k-x^*\| = o(1/k^{\Phi^\circ(\gamma,\theta) - \varepsilon}) \quad \text{where}\quad 
\Phi^\circ(\gamma,\theta) = \min\{2\gamma - \tfrac32, \tfrac{(1-\theta)(1-\gamma)}{2\theta-1}\}.\]
Due to $2\gamma-\frac32 \leq \frac32\gamma-1$, $\gamma \leq 1$, this rate is slower than our results in \Cref{coro:local-rate} (a). Moreover, \Cref{coro:local-rate} (b) allows us to further cover the case $\gamma=1$ (i.e., $\alpha_k \sim 1/k$). In this scenario, we obtain the iterate rate $o(\log^{\frac12+\varepsilon}(k)/\sqrt{k})$.

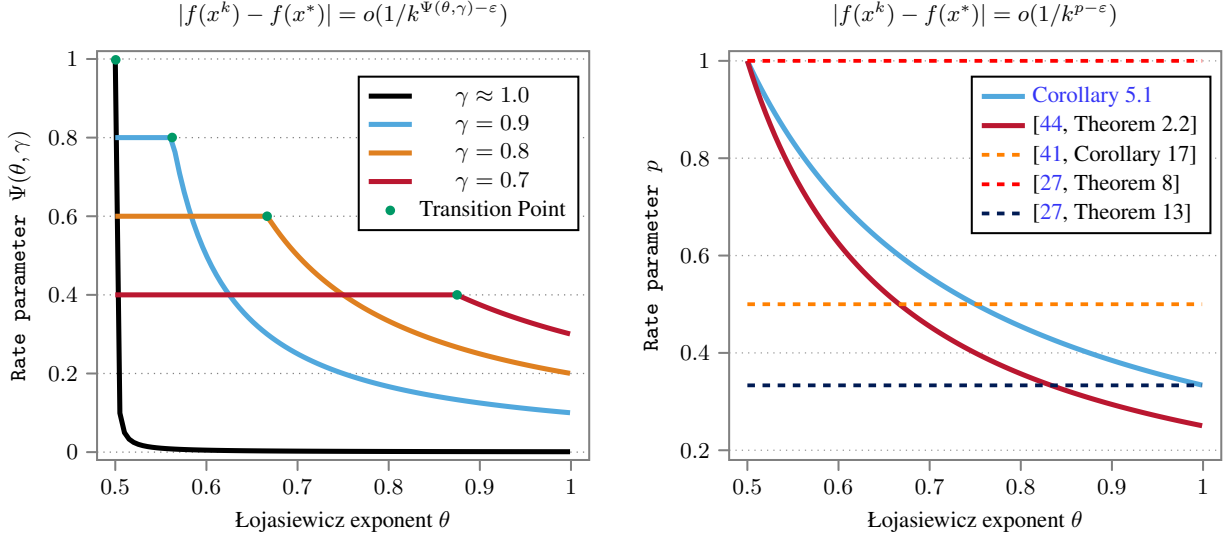
\begin{figure}[t]
\centering

\pgfmathsetmacro{\ga}{0.999}
\pgfmathsetmacro{\gb}{0.9}
\pgfmathsetmacro{\gc}{0.8}
\pgfmathsetmacro{\gd}{0.7}

\pgfmathsetmacro{\xa}{\ga/(4*\ga - 2)}
\pgfmathsetmacro{\ya}{1/(4*\xa-1)}
\pgfmathsetmacro{\xb}{\gb/(4*\gb - 2)}
\pgfmathsetmacro{\yb}{1/(4*\xb - 1)}
\pgfmathsetmacro{\xc}{\gc/(4*\gc - 2)}
\pgfmathsetmacro{\yc}{1/(4*\xc - 1)}
\pgfmathsetmacro{\xd}{\gd/(4*\gd - 2)}
\pgfmathsetmacro{\yd}{1/(4*\xd - 1)}

\begin{tikzpicture}[scale = 0.95]
\begin{axis}[
 at={(0,0)},
    domain=0.5:0.999,
    samples=100,
    title={{\small $|f(x^k) - f(x^*)| = o(1/k^{\Psi(\theta,\gamma)-\varepsilon})$}}, 
    xlabel={{\small {\L}ojasiewicz exponent $\theta$}},
    ylabel={{\small \texttt{Rate parameter $\Psi(\theta,\gamma)$}}},
    ymin=-0.02, ymax=1.02,
    xmin=0.48, xmax=1.02,
    legend pos=north east,
    ytick={0, 0.2, ..., 1},
    xtick={0.5, 0.6, 0.7, 0.8, 0.9, 1},
    restrict y to domain=0:1,
   tick align=outside,
    tickpos=left, 
    tick style={draw=black!50!, line cap=round, major tick length=4pt, minor tick length=2pt,  thick}, 
    axis line style={draw=black!50,  thick}, 
    every axis plot/.append style={thick},
    legend style={at={(.98,.94)}, draw=black, thick, font=\footnotesize},
    axis background/.style={fill=red!0},
    every axis/.append style={font=\footnotesize},
    grid=major,
    grid style={line width=.5pt, draw=gray!100, dotted},
    ymajorgrids=true,
    xmajorgrids=false,
]
\addplot[opacity=0.8, black, line width=2pt] 
    {min(2*\ga - 1, (1-\ga)/(2*x-1))};
\addplot[MyBlue, line width=2pt] 
    {min(2*\gb - 1, (1-\gb)/(2*x-1))};
\addplot[orange!50!brown, line width=2pt] 
    {min(2*\gc - 1, (1-\gc)/(2*x-1))};
\addplot[UMRed, line width=2pt] 
    {min(2*\gd - 1, (1-\gd)/(2*x-1))};
    
\addplot[green!60!blue, mark=*, only marks, mark size=1.5pt] coordinates {(\xa, \ya)};

\legend{$\gamma \approx 1.0$,$\gamma =0.9$,$\gamma =0.8$,$\gamma =0.7$, Transition Point};
\addplot[green!60!blue, mark=*, only marks, mark size=1.5pt] coordinates {(\xb, \yb)};
\addplot[green!60!blue, mark=*, only marks, mark size=1.5pt] coordinates {(\xc, \yc)};
\addplot[green!60!blue, mark=*, only marks, mark size=1.5pt] coordinates {(\xd, \yd)};

\end{axis}

\begin{axis}[
 at={(8.8cm,0)},
    domain=0.5:0.999,
    samples=100,
    title={{\small $|f(x^k) - f(x^*)| = o(1/k^{p-\varepsilon})$}},
    xlabel={{\small {\L}ojasiewicz exponent $\theta$}},
    ylabel={{\small \texttt{Rate parameter $p$}}},
    ymin=0.18, ymax=1.02,
    xmin=0.48, xmax=1.02,
    legend pos=north east,
    ytick={0, 0.2,0.4, ..., 1},
    xtick={0.5, 0.6, 0.7, 0.8, 0.9, 1},
    restrict y to domain=0:1,
   tick align=outside,
    tickpos=left, 
    tick style={draw=black!50, line cap=round, major tick length=4pt, minor tick length=2pt,  thick}, 
    axis line style={draw=black!50,  thick}, 
    every axis plot/.append style={thick},
    legend style={at={(.98,.94)}, draw=black, thick, font=\footnotesize},
    legend cell align=left,
    axis background/.style={fill=red!0},
    every axis/.append style={font=\footnotesize}, 
    grid=major,
    grid style={line width=.5pt, draw=gray!80, dotted},
    ymajorgrids=true,
    xmajorgrids=false,
]
\addplot[MyBlue, line width=2pt] {1/(4*x-1)};
\addplot[UMRed, line width=2pt] 
   {1/(6*x-2)};
\addplot[opacity=1, orange, line width=1.5pt, dashed] 
    {1/2};
\addplot[opacity=1, redp, line width=1.5pt, dashed] 
{1};
\addplot[opacity=1, OxfordBlue, line width=1.5pt, dashed] 
    {1/3};

\legend{\Cref{coro:local-rate},\cite[Theorem 2.2]{tadic2015convergence}, \cite[Corollary 17]{sebbouh2021almost},\cite[Theorem 8]{liu2022almost}, \cite[Theorem 13]{liu2022almost}};
\end{axis}
\end{tikzpicture}

\caption{Left: Function value rates in \Cref{coro:local-rate} for different $\gamma$. Right: Plot of the rate $p$ for the optimal choice of $\gamma$. The rates in \cite[Theorems 8 and 13]{liu2022almost} and \cite[Corollary 17]{sebbouh2021almost} are derived under (strong) convexity and do not depend on $\theta$.
}
\label{fig:vis}
\end{figure}
\item \emph{On the choice of $\gamma$.} When the {\L}ojasiewicz exponent $\theta$ is known, the rate functions in \Cref{coro:local-rate} (a) can be optimized in terms of $\gamma$. In particular, the optimal choice is given by $\gamma^* = \frac{2\theta}{4\theta-1}$ resulting in $\Psi(\gamma^*,\theta) = \frac{1}{4\theta-1}$, $\Phi(\gamma^*,\theta) = \frac{1-\theta}{4\theta-1}$ and $|f(x^k)-f(x^*)| = o(1/k^{\frac{1}{4\theta-1} - \varepsilon})$\footnote{The same almost sure and optimal rate for $\{f(x^k)\}_k$ was recently shown for $\SHB$ in the preprint \cite[Theorem 4.2]{weissmann2024almost}. The result in \cite{weissmann2024almost} requires the global $\theta$-Polyak-{\L}ojasiewicz condition to hold (a global variant of the {\L}ojasiewicz inequality stated in \Cref{Assumption:2}).} and $\|x^k-x^*\| = o(1/k^{\frac{1-\theta}{4\theta-1} - \varepsilon})$ for all $\varepsilon>0$.
%
%
Similarly, the optimal choice of $\gamma$ for $\Phi^\circ(\gamma,\theta)$ is $\gamma^*=\frac{4\theta-1}{6\theta-2}$ and the rates in \cite{tadic2015convergence} reduce to $|f(x^k)-f(x^*)| = o(1/k^{\frac{1}{6\theta-2} - \varepsilon})$ and $\|x^k-x^*\| = o(1/k^{\frac{1-\theta}{6\theta-2} - \varepsilon})$. In \Cref{fig:vis}, we provide a visualization of the obtained rates in different scenarios.  
\end{itemize}

\subsection{Derivations of Rates under Polynomial Step Sizes}\label{proof:coro rate}
This subsection presents the proof of \Cref{coro:local-rate}. W.l.o.g., we assume $\beta = 0$, i.e., $
    \alpha_k = \alpha/k^\gamma$. Clearly, we have $\alpha_k\to0$ and $\tg_k = \sum_{i=1}^{k}\alpha_i\to\infty$ as $k\to\infty$. To derive the claimed rates, we consider the cases $\gamma \in (\frac23,1)$ and $\gamma = 1$ separately. \\[1mm]
\noindent \textbf{Part (a).} When $\gamma \in (\frac23,1)$, we fix $\varepsilon \in (0, 3\gamma - 2)$ and set $r = \frac{2\gamma - 1}{2(1-\gamma)}-\frac{\varepsilon}{4(1-\gamma)}> \frac12$. By the integral test, we have $\tg_k =\sum_{i=1}^{k}\alpha_i= \Theta(k^{1-\gamma})$. Since $r<\frac{2\gamma-1}{2(1-\gamma)}$, it follows that $\sum_{k=1}^\infty \alpha_k^2 \tg_k^{2r} = \cO(\sum_{k=1}^\infty 1/k^{2\gamma-2(1-\gamma)r})<\infty$. 
Hence, \Cref{thm:convergence-rate} (a) is applicable. As a result, $\{\sx^k\}_k$ converges a.s.\ to some $\sx^*:\Omega\to\crit(f)$ on $\cX$ and there is $\cR\in\cF$ with $\Prob(\cR) = \Prob(\cX)$ such that the realizations $x^k \equiv \sx^k(\omega)$, $x^* \equiv \sx^*(\omega)$, $\omega\in\cR$, satisfy
 \[ \max\{|f(x^k)-f^*|,\|f(x^k)\|^2\} = \cO(1/k^{(1-\gamma)\psi(\theta)}) \quad \text{and} \quad  \|x^k-x^*\| = \cO(1/k^{(1-\gamma)\phi(\theta)}),\]
where $f^*:=f(x^*)$, $\psi(\theta)=\min\{2r,\frac{1}{2\theta-1}\}$, and $\phi(\theta)=\min\{r-\frac12,\frac{1-\theta}{2\theta-1}\}$. Next, we compare $(1-\gamma)\psi(\theta)$ with $\Psi(\gamma,\theta)=\min\{2\gamma-1,\tfrac{1-\gamma}{2\theta-1}\}$. Noting $\tfrac{1-\gamma}{2\theta-1} \geq \Psi(\gamma,\theta)$ and 
\[ 
 2r(1-\gamma) = 2\gamma-1- \varepsilon/2 > \Psi(\gamma,\theta)- \varepsilon,  \quad \text{we have} \quad 
 (1-\gamma)\psi(\theta) > \Psi(\gamma,\theta)- \varepsilon.\]
 Thus, we can re-express the rates for $|f(x^k)-f^*|$ and $\|f(x^k)\|^2$ as follows:
 \[|f(x^k)-f^*| = o(1/k^{\Psi(\gamma,\theta) - \varepsilon}) \quad \text{and} \quad \|f(x^k)\|^2= o(1/k^{\Psi(\gamma,\theta) - \varepsilon}) .
	 \]
Similarly, we can derive $\|x^k-x^*\| = o(1/k^{\Phi(\gamma,\theta) - \varepsilon})$ by showing that $(1-\gamma)\phi(\theta) > \Phi(\gamma,\theta)- \varepsilon$, where $\Phi(\gamma,\theta)=\min\{\frac32\gamma-1,\tfrac{(1-\gamma)(1-\theta)}{2\theta-1}\}$.\\[-2mm]

  \noindent \textbf{Part (b).} When $\gamma=1$, for any $\varepsilon>0$, we define $g(x) = x^{-p} \cdot \exp(\frac{x}{2\alpha})$ with $p = \frac12 + \frac{\varepsilon}{4}$. By the integral test, we obtain $\alpha \log(k) \leq \tg_k \leq \alpha(1+\log(k))$ and
   	\[
	{\sum}_{k=1}^{\infty}\;\alpha_k^2 \cdot  g(\tg_k)^2 \leq \alpha^{2-2p}  \exp(1) \cdot {\sum}_{k=1}^{\infty}\; 1/(k\cdot\log(k)^{2p}) <\infty.
	\] 
The result is due to \Cref{thm:convergence-rate} (b) and $\frac{1}{g(\tg_k)} \leq \frac{\alpha^p[1+\log(k)]^p}{\sqrt{k}}=o({\log(k)^{\frac{1+\varepsilon}{2}}}/{\sqrt{k}})$. 

\newpage
\bibliographystyle{siam}
\bibliography{references} 

\newpage
\appendix
\addtocontents{toc}{\protect\setcounter{tocdepth}{1}}

\section{Proof of Key Lemmas}
To guide the derivations in the following sections, we provide a summary of the main notations in \Cref{table2}.

\begin{table}[t]
\centering
{
\NiceMatrixOptions{cell-space-limits=2pt}
{\small
\begin{NiceTabular}{|p{2.45cm}|p{13.4cm}|}%
 [ 
   code-before = 
    \rectanglecolor{lavender!40}{3-1}{3-2}
    \rectanglecolor{lavender!40}{6-1}{6-2}
    \rectanglecolor{lavender!40}{8-1}{8-2}
 ]
\toprule
\Block[c]{1-1}{\textbf{Notation}} & \Block[c]{1-1}{\textbf{Description}}  \\\Hline   
\Block{1-1}{$\mer$} & \Block[l]{1-1}{merit function $\mer(x,z) := f(z) + \zeta \|z-x\|^2$ with $\zeta := \frac{3\sL}{1-\lambda}$;}  \\
\Block{1-1}{$\tg_{m,n}$} & \Block[l]{1-1}{natural time scale $\tg_{m,n}=\sum_{i=m}^{n-1}\alpha_i$ for $n> m$;} \\[1mm]
\Block{2-1}{$\tw$, $\ti_k$, $\Ti_k$} & \Block[l]{1-1}{time window; associated time indices satisfying $\tg_{\ti_k,\ti_{k+1}} \leq \tw$; time indices collections} \\[-1mm]
 & \Block[l]{1-1}{$\Ti_k:=\{t\in\N:\ti_k<t\leq \ti_{k+1}\}$;} \\
\Block{1-1}{$\se^k$} & \Block[l]{1-1}{stochastic approximation errors $\se^k := \nabla f(\tilde \sx^k) - \sg^k$;} \\
\Block{1-1}{$\scs_k$} & \Block[l]{1-1}{aggregation of stochastic errors $\scs_k := \max_{\ell\in\Ti_k}\|\sum_{i=\gamma_k}^{\ell-1} \alpha_i\se^i\|$;} \\
\Block{1-1}{$\sd_{k}$} & \Block[l]{1-1}{iterate distance $\sd_{k} := \max\{ \max_{\ell\in\Ti_k}\|\sx^\ell-\sx^{\ti_k}\|,  \max_{\ell\in\Ti_k}\|\sz^\ell-\sz^{\ti_k}\|\}$;} \\
\bottomrule
\end{NiceTabular}
}
}
\caption{List of variables, parameters, and functions.}
  \label{table2}
  \vspace{-4mm}
\end{table}
\subsection{Proof of \texorpdfstring{\Cref{lem:err_estimate}}{Lemma 3.2}} \label{proof:martingale}
The proof of \Cref{lem:err_estimate} relies on standard martingale techniques and is based on the Burkholder-Davis-Gundy inequality, \cite{BurDavGun72,Str11}.
\begin{lemma}[Burkholder-Davis-Gundy Inequality]
		\label{Thm:BDG}
		Let $\{\sw^{k}\}_k$ be a vector-valued martingale with an associated filtration $\{\mathcal U_k\}_k$ and $\sw^{0}=0$. For all $p \in (1,\infty)$, there then is $C_p > 0$ such that
		\[ \Exp\left[{{\sup}_{k \geq 0} \|\sw^k\|^p}\right] \leq C_p \cdot \Exp\Big[\Big( {\sum}_{k=1}^\infty \|\sw^{k}-\sw^{k-1}\|^2 \Big)^{{p}/{2}} \Big]. \]
\end{lemma}
We now start with the proof of \Cref{lem:err_estimate}. 
\begin{proof}{Proof}  Let us define $\mathcal{U}_\ell:=\mathcal{F}_{\ti_k+\ell}$ and let the sequence $\{\sy^\ell\}_\ell$ be given by $\sy^0:=0$, $\sy^\ell := {\sum}_{i=\ti_k}^{\min\{\ti_k+\ell,\ti_{k+1}\}-1}\alpha_i\beta_i \se^i$ for all $\ell \geq 1$.
Then, each of the functions $\sy_\ell$ is $\mathcal U_\ell$-measurable and for all $\ell \in \{1,\dots,\ti_{k+1}-\ti_k-1\}$, we have 
\[ {\Exp[{\sy^{\ell+1} \mid \mathcal U_\ell}] = {\sum}_{i=\ti_k}^{\ti_k+\ell} \alpha_i\beta_i \Exp[{\se^i \mid \mathcal U_\ell}] = \sy^\ell + \alpha_{\ti_k+\ell}\beta_{\ti_k+\ell} \Exp[{\se^{\ti_k+\ell} \mid \mathcal{F}_{\ti_k+\ell}}] = \sy^\ell.} \]
(similarly for $\ell \geq \ti_{k+1}-\ti_k$). Thus, $\{\sy^\ell\}_\ell$ defines a $\{\mathcal{U}_\ell\}$-martingale. By \Cref{Thm:BDG}, \Cref{as:sgd},  \eqref{step size-1}, and setting $\bar\scs_k :=\max_{j \in \Ti_k} \| {\sum}_{i=\ti_k}^{j-1} \alpha_i \beta_i \se^i\|$, it follows
%
\begin{equation}\label{eq:app-bdg}
    \begin{aligned}
        \Exp[{\bar\scs_k^2}] =\Exp\left[{{\sup}_{\ell>0} \|\sy^\ell\|^2}\right] & \leq C_2 \cdot  \Exp\left[{ {\sum}_{\ell=1}^{\infty}\|\sy^\ell-\sy^{\ell-1}\|^2}\right] \\ & = C_2 {\sum}_{i=\ti_k}^{\ti_{k+1}-1}\alpha_i^2 \beta_i^2 \Exp[{ \|\se^i\|^2}]  \leq C_2 \sigma^2 {\sum}_{i=\ti_k}^{\ti_{k+1}-1} \alpha_i^2\beta_i^2 < \infty. 
    \end{aligned}
\end{equation}
%
Next, for all $j \in \Ti_k$ and similar to \cite[Lemma 6.1]{tadic2015convergence}, we have
\[
\begin{aligned} {\sum}_{i= \ti_k}^{j-1} \alpha_i \se^i & = \beta_{j-1}^{-1} {\sum}_{i=\ti_k}^{j-1} \alpha_i\beta_i \se^i -  \beta_{j-1}^{-1} {\sum}_{i=\ti_k}^{j-2} \alpha_i \beta_i \se^i + {\sum}_{i=\ti_k}^{j-2} \alpha_i \se^i \\ & = \dots =  \beta_{j-1}^{-1} {\sum}_{i=\ti_k}^{j-1} \alpha_i\beta_i \se^i + {\sum}_{\ell=\ti_k}^{j-2}  (\beta_\ell^{-1}-\beta_{\ell+1}^{-1}) {\sum}_{i = \ti_k}^\ell \alpha_i \beta_i \se^i.  \end{aligned}
\]
Since $\{\beta_k\}_k$ is assumed to be non-decreasing, this yields
\[
\begin{aligned}
	 \scs_k =  \max_{j\in\Ti_k } \left\| {\sum}_{i=\ti_k}^{j-1} \alpha_i \se^i\right\| \leq \max_{j\in\Ti_k } \left[ \beta_{j-1}^{-1}+ {\sum}_{\ell=\ti_k}^{j-2}(\beta_\ell^{-1}-\beta_{\ell+1}^{-1}) \right] \cdot \bar\scs_k = \beta_{\ti_k}^{-1} \cdot \bar\scs_k, 
\end{aligned}
\]
for all $k \geq 1$. The monotone convergence theorem and \eqref{eq:app-bdg} imply $\Exp [{ {\sum}_{k=1}^\infty \beta_{\ti_k}^2{\scs_k^2}}] = {\sum}_{k=1}^\infty {\Exp[ { \beta_{\ti_k}^2\scs_k^2 }] } \leq  {\sum}_{k=1}^{\infty}\Exp[\bar\scs_k^2] < \infty$.
Hence, we obtain $\sum_{k=1}^\infty \beta_{\ti_k}^2{\scs_k^2} < \infty$ a.s..
\end{proof}


\subsection{Proof of \texorpdfstring{\Cref{lem:iterate-bound}}{Lemma 3.3}} \label{Proof:lem:iterate-bound}
\begin{proof}{Proof}
Let us  pick any $\delta \in [0,1)$ and $\tw \in (0,\frac{1-\lambda}{\sL\tau}]$ where $\tau:=\frac{20(1+2\mparam)}{1-\lambda}$. According to $\alpha_k\to0$ and  \Cref{lem:time-length}, there is $ K_\tw \geq 1$, such that for all $k\geq K_\tw$, it holds that
\begin{equation}\label{eq:step size requirement}
    \begin{aligned}
    \nonumber &{\sum}_{i=\ti_k}^{\ell-1}\alpha_i \leq \tg_{\ti_k,\ti_{k+1}} \leq \tw \quad  \text{for all } \ell \in \Ti_k=\{t\in\N : \ti_k < t \leq \ti_{k+1}\} \\
    &  \sL \mparam  \alpha_{\ti_k} \leq \lambda \iota \quad \text{where}\quad \iota:=\frac{1}{10} \cdot \min\Big\{\frac{1}{10},\frac{\mparam (1-\lambda)}{1+2\mparam}\Big\}. 
\end{aligned}  
\end{equation}
  %
%
To simplify the notations, we denote $m:=\ti_k$ and $n:=\ti_{k+1}$.\\[2mm]
\noindent\textbf{Step 1:} \emph{Bounding $\max_{\ell\in\Ti_k}  \|\sx^\ell-\sx^{m}\|^2$.} By the update of $\SGDM$ in \cref{algo:sgdm}, 
\[ \sx^\ell - \sx^m = \lambda (\sx^{\ell-1} - \sx^{m-1}) - \sum_{i=m}^{\ell-1}\alpha_i \sg^i = \lambda (\sx^{\ell-1} - \sx^{m}) + \lambda (\sx^{m} - \sx^{m-1}) - \sum_{i=m}^{\ell-1}\alpha_i \sg^i
\]
holds for all $\ell\in\Ti_k$. Setting $\sy^\ell := \lambda^{-\ell}(\sx^\ell - \sx^m)$ and $\sct{\eta}^\ell:=\lambda^{-\ell}[(\sx^m-\sx^{m-1}) - \lambda^{-1}\sum_{i=m}^{\ell}\alpha_i \sg^i]$, we can rewrite the above expression as $\sy^\ell = \sy^{\ell-1} + \sct{\eta}^{\ell-1}$. Unfolding this recursion, it follows $\sy^\ell= \sum_{j=m}^{\ell-1} \sct{\eta}^j$ and
\[
\begin{aligned}
    \sx^\ell - \sx^m = \lambda^\ell \sy^\ell &= {\sum}_{j=m}^{\ell-1}\lambda^{\ell-j}\Big[(\sx^m-\sx^{m-1}) - \lambda^{-1}{\sum}_{i=m}^{j}\alpha_i\sg^i \Big] \\
    &=\frac{\lambda(1-\lambda^{\ell-m})}{1-\lambda}(\sx^m-\sx^{m-1}) - {\sum}_{j=0}^{\ell-m-1}\lambda^j\,{\sum}_{i=m}^{\ell-j-1}\alpha_i \sg^i.
\end{aligned}
\]
Noting $\sg^i=\nabla f(\nx^i) - \se^i = \nabla f(\sx^m) - \se^i + [\nabla f(\nx^i) -  \nabla f(\sx^m)]$, we obtain
\begin{equation}\label{eq:lem-sgdm-important}
\begin{aligned}
   \sx^\ell - \sx^m &= \frac{\lambda(1-\lambda^{\ell-m})}{1-\lambda}(\sx^m-\sx^{m-1}) - {\sum}_{j=0}^{\ell-m-1}\lambda^j\;{\sum}_{i=m}^{\ell-j-1}\alpha_i \nabla f(\sx^m) \\
   &\hspace{3ex}+{\sum}_{j=0}^{\ell-m-1}\lambda^j\;{\sum}_{i=m}^{\ell-j-1} \Big[\alpha_i \se^i + \alpha_i \big(\nabla f(\sx^m)-\nabla f(\nx^i) \big) \Big].
\end{aligned}
\end{equation}
Taking norm in \eqref{eq:lem-sgdm-important} and using $\sum_{j=0}^{\ell-m-1}\lambda^j \leq \frac{1}{1-\lambda}$ and $\sum_{i=m}^{\ell-1}\alpha_i  \leq \tw$, this yields 
\begin{equation}
    \label{eq:def sW}
    \begin{aligned}
 (1-\lambda)\|\sx^\ell - \sx^m\|  & \leq \lambda\|\sx^m - \sx^{m-1}\| + \tw \|\nabla f(\sx^m)\|  \\ 
&\hspace{4ex} + \max_{m<j\leq \ell}\Big\|{\sum}_{i=m}^{j-1}\alpha_i \se^i\Big\| + {\sum}_{i=m}^{n-1}\alpha_i \|\nabla f(\nx^i)- \nabla f(\sx^m)\| \\
& \leq \lambda\|\sx^m - \sx^{m-1}\| + \tw \|\nabla f(\sx^m)\| + \scs_{k} + \sL{\sum}_{i=m}^{n-1}\alpha_i \|\nx^i -\sx^m\| =:(1-\lambda)\sW,
\end{aligned}
\end{equation}
where we applied $\max_{m<j\leq \ell}\|\sum_{i=m}^{j-1}\alpha_i \se^i\| \leq \scs_{k}$ and \Cref{Assumption:1} to obtain the second inequality. Due to $\|\nx^i - \sx^m\| \leq(1+\mparam)\|\sx^i - \sx^m\| + \mparam \|\sx^{i-1} - \sx^m\|$ and since $\{\alpha_k\}_k$ is non-increasing, it further follows
\[ {\sum}_{i=m}^{n-1}\alpha_i \|\nx^i -\sx^m\| \leq  \mparam \alpha_m \|\sx^m - \sx^{m-1}\| + (1+2\mparam){\sum}_{i=m}^{n-1}\alpha_i \|\sx^i -\sx^m\|. \]
Thus, combining this with \eqref{eq:step size requirement} and defining $\sV:={\sum}_{i=m}^{n-1}\alpha_i \|\sx^i-\sx^m\|$, we can infer 
\[ (1-\lambda)\sW \leq (1+\iota)\lambda\|\sx^m - \sx^{m-1}\| + \tw \|\nabla f(\sx^m)\| + \scs_{k} + \sL(1+2\mparam)\sV. \]
Next, using $(1-\lambda)(\sz^k - \sx^k) =\lambda(\sx^k - \sx^{k-1})$ (cf$.$ \eqref{eq:lem-iter-bound-z-x}) and $\|\nabla f(\sx^m)\|\leq \|\nabla f(\sz^m)\| + \sL\|\sz^m-\sx^m\|$, it follows
%
\begin{equation}\label{eq:lem-sgdm-bounded-update-1}
    \begin{aligned}
     (1-\lambda)\sW  & \leq  (1+\iota)(1-\lambda)\|\sz^m - \sx^{m}\| + \tw \|\nabla f(\sx^m)\| + \scs_{k}  + \sL(1+2\mparam)\sV \\ 
    &\leq [(1+\iota)(1-\lambda)+\sL\tw]\|\sz^m - \sx^{m}\| + \tw \|\nabla f(\sz^m)\| + \scs_{k}  + \sL(1+2\mparam)\sV \\
    &\leq \underbracket{(1-\lambda)[1+\iota+\tau^{-1}] \|\sz^m - \sx^{m}\| + {\tw}\|\nabla f(\sz^m)\| + {\scs_k}}_{=:(1-\lambda)\sv} + {\sL(1+2\mparam)\sV}, 
\end{aligned}
\end{equation}
where the last line is due to $\sL\tw \leq (1-\lambda)/\tau$.
Since $\sV={\sum}_{i=m}^{n-1}\alpha_i \|\sx^i-\sx^m\|$ and $\max_{m\leq i\leq n}\|\sx^i-\sx^m\| \leq \sW$ according to \eqref{eq:def sW}, it holds that $\sV \leq ({\sum}_{i=m}^{n-1}\alpha_i) \sW \leq (\sum_{i = m}^{n-1} \alpha_i) \cdot [\frac{(1+2\mparam)\sL}{1-\lambda}\sV + \sv]$, where the last inequality follows from \eqref{eq:lem-sgdm-bounded-update-1}. Rearranging this estimate and utilizing $\sum_{i = m}^{n-1} \alpha_i \leq \tw \leq \frac{1-\lambda}{\sL\tau}$, this yields
\begin{equation}
    \label{eq:lem:bound V W}
    \sV \leq \Big[1-\frac{1+2\mparam}{\tau}\Big]^{-1} \frac{1-\lambda}{\sL\tau} \sv \leq \frac{1-\lambda}{\sL(\tau-1-2\mparam)} \sv \quad \text{and} \quad \sW  \leq \frac{\tau}{\tau-1-2\nu}\sv.
    \end{equation}
Invoking $\tau = \frac{20(1+2\nu)}{1-\lambda}$, it further follows $\frac{\tau}{\tau-1-2\nu} \leq \frac{20}{19}$ and hence, we obtain
\begin{equation}
    \label{eq:lem:bound-state a}
\max_{\ell\in\Ti_k}  \|\sx^\ell-\sx^m\| \leq \sW \leq  \frac{20}{19} \Big[\Big(1+\iota + \frac{1}{\tau}\Big)\|\sz^m - \sx^{m}\| + \frac{\tw\|\nabla f(\sz^m)\|+\scs_k}{1-\lambda}\Big].
\end{equation}
Finally, using $(a+b+c)^2 \leq (1+2/\varepsilon)a^2 + (2+\varepsilon)(b^2+c^2)$ with $\varepsilon=10$, $\frac{400\cdot 12}{361} \leq 15$, and $(1+\frac15)\frac{400}{361}(1+\iota+\tau^{-1})^2 \leq \frac43 (\frac{53}{50})^2 \leq \frac32$, (due to $\tau \geq 20$ and $\iota \leq \frac{1}{100}$), we have 
\begin{equation}
    \label{eq:lem:bound-state a 2}
\max_{\ell\in\Ti_k}  \|\sx^\ell-\sx^m\|^2 \leq \sW^2 \leq \frac{3}{2} \|\sz^m - \sx^{m}\|^2 + \frac{15[\tw^2 \|\nabla f(\sz^m)\|^2 + \scs_{k}^2]}{(1-\lambda)^2}.
\end{equation}

\noindent\textbf{Step 2:} \emph{Bounding $\max_{\ell\in\Ti_k}\|\sz^\ell - \sz^{m}\|$.}
Based on \eqref{eq:lem-iter-bound-z-x}, it holds for all $\ell\in\Ti_k$ that
\begin{equation}\label{eq:recursion-z}
    \begin{aligned}
        (1-\lambda)(\sz^\ell - \sz^m) & = - {\sum}_{i=m}^{\ell-1} \alpha_i \sg^i \\ &=  - \tg_{m,\ell} \nabla f(\sx^m) + {\sum}_{i=m}^{\ell-1} \alpha_i \se^i - {\sum}_{i=m}^{\ell-1} \alpha_i[\nabla f(\nx^i) -  \nabla f(\sx^m)], 
        \end{aligned}
\end{equation}    
    which yields
    \begin{equation}
        \label{eq:lem:bound-state b}
             \max_{\ell\in\Ti_k}\|\sz^\ell - \sz^m\| \leq \frac{1}{1-\lambda} \Big[\tw\|\nabla f(\sx^m)\| + \scs_k + \sL {\sum}_{i=m}^{n-1}\alpha_i \|\nx^i-\sx^m\| \Big] \leq \sW,
    \end{equation}
    where the last step is due to the definition of $\sW$ in  \eqref{eq:def sW}. Taking square on both sides of \eqref{eq:lem:bound-state b} and using \eqref{eq:lem:bound-state a 2}, this establishes the first statement \eqref{eq:tedious-01-later} of \Cref{lem:iterate-bound}.\\[1mm]   
\textbf{Step 3:} \emph{Bounding $\|\sx^{n} - \sx^{n-1}\|^2$.}
W.l.o.g., we assume $\lambda \neq 0$ and obtain 
\[
     \frac{\sx^n-\sx^{n-1}}{\lambda^n} = \frac{\sx^{n-1}-\sx^{n-2}}{\lambda^{n-1}} - \frac{\alpha_{n-1}\sg^{n-1}}{\lambda^{n}} = \cdots = \frac{\sx^m-\sx^{m-1}}{\lambda^m} - {\sum}_{i=m}^{n-1}\; \frac{\alpha_{i}\sg^{i}}{\lambda^{i+1}}.
\]
Substituting $\sg^i= \nabla f(\sx^m) - \se^i + [\nabla f(\nx^i) -  \nabla f(\sx^m)]$, using the triangle inequality, and setting $\tilde \tau := (1+2\mparam)^{-1}\tau$ and $\bar\sr:=\sum_{i=m}^{n-1}\alpha_i\lambda^{-i}\se^i$, we obtain
\begin{equation}\label{eq:lem-sgdm-bounded-update-2}
    \begin{aligned}
         \|\sx^n-\sx^{n-1}\| & \leq  \lambda^{n-m}\|\sx^m-\sx^{m-1}\| + \lambda^{n-1}\Big\|{\sum}_{i=m}^{n-1} \alpha_{i}\lambda^{-i} \sg^{i} \Big\|\\ 
    &\leq  \lambda\|\sx^m-\sx^{m-1}\| +  {\sum}_{i=m}^{n-1} \alpha_{i} [\|\nabla f(\sx^m)\| + \sL\|\nx^i-\sx^m\|] + \lambda^{n-1} \|\bar\sr\| \\
     &\leq (1-\lambda) \sW - \scs_k +\lambda^{n-1}\|\bar\sr\|, 
    \end{aligned}
\end{equation}
where the second line utilizes $\lambda^{n-1}/\lambda^i \leq 1$ and the $\sL$-smoothness of $f$, and the last line follows from $\sum_{i=m}^{n-1} \alpha_i \leq \tw$ and \eqref{eq:def sW}. Setting $\sr^\ell:=\sum_{i=m}^{\ell-1}\alpha_i \se^i$, we then obtain
%
\[\begin{aligned}
 \bar\sr  & = \lambda^{1-n}(\sr^n - \sr^{n-1}) + {\sum}_{i=m}^{n-2} \alpha_{i}\lambda^{-i}\se^i= \cdots \\
    & = {\sum}_{i=m+2}^n\lambda^{1-i}(\sr^i - \sr^{i-1}) + \lambda^{-m}\sr^{m}= \lambda^{1-n}\sr^n -  (1-\lambda){\sum}_{i=m+1}^{n-1} \lambda^{-i}\sr^i.
\end{aligned}\]
%
Noting $\|\sr^\ell\| \leq \scs_k$ for all $\ell \in \Ti_k$ and using $\sum_{i=0}^{n-1}\lambda^{-i} = \frac{\lambda^{1-n}-\lambda}{1-\lambda}$, this yields
\[ \|\bar\sr\| = \Big\|{\sum}_{i=m}^{n-1} \alpha_{i}\lambda^{-i} \se^{i} \Big\| \leq \lambda^{1-n}\scs_k + (1-\lambda) \scs_k {\sum}_{i=m+1}^{n-1} \lambda^{-i} \leq 2\lambda^{1-n}\scs_k.
\]
Recalling $\tilde \tau = \frac{\tau}{1+2\mparam}$, $\tau = \frac{20(1+2\nu)}{1-\lambda}$ and using \eqref{eq:lem:bound V W} with $\frac{\tau}{\tau-1-2\nu} = \frac{\tilde \tau}{\tilde \tau - 1}$ and the relation $(1-\lambda)\sv=  (1+\iota+\tau^{-1})\lambda\|\sx^m - \sx^{m-1}\| + \tw\|\nabla f(\sz^m)\| + \scs_k$ (cf$.$ \eqref{eq:lem-sgdm-bounded-update-1}, \eqref{eq:lem-iter-bound-z-x}), we have
\[\begin{aligned}
    \|\sx^n-\sx^{n-1}\| &\leq (1-\lambda) \sW + \scs_k\\ 
    &\leq  \frac{\tilde \tau}{\tilde\tau-1} (1+\iota+\tau^{-1})\lambda \cdot \|\sx^m - \sx^{m-1}\|  + \frac{\tilde\tau}{\tilde\tau-1} [\tw \|\nabla f(\sz^m)\|+ 2\scs_k] \\
    &\leq \frac{\tilde\tau+1}{\tilde\tau-1}\lambda \cdot \|\sx^m - \sx^{m-1}\| + \frac{\tilde\tau}{\tilde\tau-1} [\tw \|\nabla f(\sz^m)\|+ 2\scs_k],
\end{aligned}\]
where the last line is due to $\iota \leq \frac{\mparam (1-\lambda)}{10(1+2\mparam)} = 2\mparam \tau^{-1}$ and $\iota+\tau^{-1} \leq (2\mparam + 1)\tau^{-1} = {\tilde \tau}^{-1}$. Thus, applying $(a+b+c)^2 \leq (1+2/\varepsilon)a^2 + (2+\varepsilon)(b^2+c^2)$ with  $\varepsilon=\frac{2(2\lambda+1)}{1-\lambda}$ to the last estimate, it holds that
\begin{equation}\label{eq:lem-sgdm-bounded-update-3}
\|\sx^n-\sx^{n-1}\|^2 \leq c_1 \|\sx^m - \sx^{m-1}\|^2 + c_2 [\tw^2\|\nabla f(\sz^m)\|^2 + 4\scs_k^2].
\end{equation}
where $c_1 = (\frac{\tilde\tau+1}{\tilde\tau-1})^2\frac{2+\lambda}{2\lambda+1}\lambda^2$ and $c_2 = \frac{2(2+\lambda)}{1-\lambda}(\frac{\tilde\tau}{\tilde\tau-1})^2$. Moreover, we have $\frac{\tilde\tau+1}{\tilde\tau-1} = 1+ \frac{2}{\tilde\tau-1}$ with $\frac{2}{\tilde\tau-1}= \frac{2(1-\lambda)}{19+\lambda}$, $\frac{2}{\tilde\tau-1} \leq \frac{1}{9}$ and $\frac{2}{\tilde\tau-1} \leq \frac{1-\lambda}{8} = \frac{9}{40}\frac{5(1-\lambda)}{3\cdot 3} \leq \frac{9(1-\lambda)(4\lambda+5)}{38(\lambda+2)(2\lambda+1)}$, based on the definition of $\tilde\tau$ and $\tau$. Next, utilizing $\lambda \leq \frac{2\lambda+1}{3}$, we obtain
\[
    c_1 \leq \Big(1+\frac{38}{9(\tilde\tau-1)}\Big)\frac{(2+\lambda)\lambda^2}{(2\lambda+1)} \leq \Big[1+\frac{(1-\lambda)(4\lambda+5)}{2(2+\lambda)(2\lambda+1)}\Big]\cdot\frac{(\lambda+2)(2\lambda+1)}{9} = \frac{1+\lambda}{2},
\]
where we first applied $(1+ \frac{2}{\tilde\tau-1})^2 = 1 + \frac{2}{\tilde\tau-1}(2+ \frac{2}{\tilde\tau-1}) \leq 1 + \frac{38}{9(\tilde\tau-1)}$. In addition, due to $(\frac{\tilde\tau}{\tilde\tau-1})^2 = (\frac{20}{19+\lambda})^2 \leq (\frac{20}{19})^2 \leq \frac43$, it holds that $c_2 \leq \frac{8}{1-\lambda}$ and, by \eqref{eq:lem-sgdm-bounded-update-3}, we can infer
\[
     \|\sx^n-\sx^{n-1}\|^2 \leq \frac{1+\lambda}{2}\|\sx^m - \sx^{m-1}\|^2 + \frac{8}{1-\lambda}[\tw^2\|\nabla f(\sz^m)\|^2 + 4\scs_k^2].
\]
Finally, using $\sz^k - \sx^k = \frac{\lambda}{1-\lambda}(\sx^k - \sx^{k-1})$, this completes the proof of \eqref{eq:tedious-02-later}.
\end{proof}

\subsection{Proof of \texorpdfstring{\Cref{lem:approx-desent}}{Lemma 3.4}}\label{proof:lem:approx-desent}
\begin{proof}{Proof}
Let us set $\delta = 0.99$ and pick $\tw \in (0,\frac{(1-\lambda)^3}{50\sL(1+2\mparam)^2}] \subset (0,\frac{(1-\lambda)^2}{20\sL(1+2\mparam)}]$. Then, \Cref{lem:time-length,lem:iterate-bound} are applicable and there is $ K_\delta \geq 1$, such that for all $k\geq K_\delta$,
\begin{equation}
    \label{eq:lem:descent-0}
    \begin{aligned}
        &\delta \tw \leq \tg_{\ti_k,\ti_{k+1}} \leq \tw, \quad
        \sL \mparam^2 \alpha_{\gamma_k} \leq 
        \lambda^2/80, \quad\text{and} \quad \eqref{eq:tedious-01-later}\;\&\; \eqref{eq:tedious-02-later}\;\,\text{hold}. 
    \end{aligned}
\end{equation}
We denote $m:=\ti_k$ and $n:=\ti_{k+1}$. First, using the $\sL$-continuity of $\nabla f$, it holds that
\begin{equation}
\label{eq:sgdm-descent-prop-1}
 f(\sz^n)  \leq f(\sz^m) + \langle \nabla f(\sz^m),\sz^n-\sz^m \rangle + \frac{\sL}{2}\sd_k^2. 
\end{equation}
Next, applying \eqref{eq:recursion-z}, we obtain $(1-\lambda)\langle \nabla f(\sz^m),\sz^n-\sz^m \rangle = \tg_{m,n}T_1 + T_2 + T_3$, where
\[ T_1 := - \langle\nabla f(\sz^m),\nabla f(\sx^m)\rangle , \quad T_2 := {\sum}_{i=m}^{n-1}\alpha_i \langle \nabla f(\sz^m),\nabla f(\sx^m)-\nabla f(\nx^i) \rangle, \]
and $T_3 := \langle \nabla f(\sz^m),{\textstyle\sum}_{i=m}^{n-1} \alpha_i \se^i \rangle$. By Young's inequality and for any $\varepsilon_1,\varepsilon_2>0$, we further have $T_1 = -\|\nabla f(\sz^m)\|^2 + \langle \nabla f(\sz^m),\nabla f(\sz^m)-\nabla f(\sx^m) \rangle \leq -(1-\frac{\varepsilon_1}{2})\|\nabla f(\sz^m)\|^2 + \frac{\sL^2}{2\varepsilon_1}\|\sz^m-\sx^{m}\|^2$ and  
\[
\begin{aligned}
    T_2 &  \leq {\sum}_{i=m}^{n-1}\alpha_i \|\nabla f(\sz^m)\| \cdot \sL[(1+\mparam)\|\sx^m- \sx^i\| + \mparam\|\sx^m- \sx^{i-1}\| ]\\
    & \leq {\sum}_{i=m}^{n-1}\alpha_i [\varepsilon_2\|\nabla f(\sz^m)\|^2 + \tfrac{\sL^2(1+\mparam)^2}{2\varepsilon_2} \|\sx^m -\sx^i\|^2 + \tfrac{\sL^2\mparam^2}{2\varepsilon_2} \|\sx^m -\sx^{i-1}\|^2] \\ 
    &  \leq \varepsilon_2\tg_{m,n}\|\nabla f(\sz^m)\|^2 + \tfrac{(1+2\mparam)^2\sL^2}{2\varepsilon_2} \tg_{m,n}\sd_k^2 + \tfrac{\sL^2\mparam^2}{2\varepsilon_2} \alpha_m\|\sx^m -\sx^{m-1}\|^2.
\end{aligned}
\]
By the definition of $\scs_k$, it holds that
$T_3 \leq \frac{\varepsilon_3\tg_{m,n}}{2}\|\nabla f(\sz^m)\|^2 + \frac{\scs_k^2}{2\varepsilon_3\tg_{m,n}}$ for all $\varepsilon_3>0$. 
Thus, combining the previous estimates, \eqref{eq:lem-iter-bound-z-x}, and $\sL \mparam^2 \alpha_{m} \leq {\lambda^2}/{80}$, we obtain
\[\begin{aligned}
    & \hspace{-2ex} \langle \nabla f(\sz^m),\sz^n-\sz^m \rangle + (1- \tfrac{\varepsilon_1+2\varepsilon_2+\varepsilon_3}{2})\tfrac{\tg_{m,n}}{1-\lambda}\|\nabla f(\sz^m)\|^2 \\ &\leq  \big[\tfrac{\sL^2\tg_{m,n}}{2\varepsilon_1(1-\lambda)} +  \tfrac{\sL}{160\varepsilon_2} \big]\|\sz^m-\sx^{m}\|^2 + \tfrac{(1+2\mparam)^2\sL^2\tg_{m,n}}{2\varepsilon_2(1-\lambda)}\cdot\sd_k^2 + \tfrac{1}{2\varepsilon_3(1-\lambda)\tg_{m,n}}\cdot \scs_k^2.
\end{aligned}\]
Setting $\varepsilon_1 = \frac{1}{5}$, $\varepsilon_2 = \frac{1}{8}$, $\varepsilon_3 = \frac{1}{9}$, recalling $\delta=0.99$, and noticing $1-\frac{\varepsilon_1+2\varepsilon_2+\varepsilon_3}{2} = \frac{259}{360} \geq \frac{7}{10\delta}$, it follows
\[\begin{aligned}
    & \hspace{-2ex} \langle \nabla f(\sz^m),\sz^n-\sz^m \rangle + \tfrac{7\tw}{10(1-\lambda)}\|\nabla f(\sz^m)\|^2 \\ &\leq  \big[\tfrac{5\sL^2 \tw}{2(1-\lambda)} +  \tfrac{\sL}{20}\big] \|\sz^m-\sx^{m}\|^2 + \tfrac{4(1+2\mparam)^2\sL^2\tw}{1-\lambda}\cdot\sd_k^2 + \tfrac{9}{2(1-\lambda)\delta \tw}\cdot \scs_k^2,
\end{aligned}\]
where we applied $\delta \tw \leq \tg_{m,n} \leq \tw$ (cf$.$ \eqref{eq:lem:descent-0}). Plugging this estimate into \eqref{eq:sgdm-descent-prop-1} and using $\tw \leq \frac{1-\lambda}{50\sL(1+2\mparam)^2} \leq \frac{1-\lambda}{50\sL}$ and $\frac{4\sL}{50} + \frac{\sL}{2} \leq \frac{7\sL}{12}$, it holds that
\begin{equation}
    \label{eq:lem:descent-1}
    f(\sz^n) + \tfrac{7\tw}{10(1-\lambda)}  \|\nabla f(\sz^m)\|^2 \leq f(\sz^m) + \tfrac{7\sL}{12}\sd_k^2 + \tfrac{\sL}{10}\|\sz^m-\sx^{m}\|^2 +  \tfrac{5}{(1-\lambda)\tw}\cdot \scs_k^2.
\end{equation}
Adding $\frac{3\sL}{1-\lambda}\|\sz^n-\sx^n\|^2 + \frac{\sL}{12}\sd_k^2$ on both sides of \eqref{eq:lem:descent-1}, we can infer
%
\begin{equation}
\label{eq:lem:descent-2}
    \begin{aligned} 
    & \hspace{-2ex} f(\sz^n) + \tfrac{3\sL}{1-\lambda}\|\sz^n-\sx^n\|^2  + \tfrac{\sL}{12}\sd_k^2 + \tfrac{7\tw}{10(1-\lambda)}\|\nabla f(\sz^m)\|^2\\
     &\leq f(\sz^m) + \tfrac{2\sL}{3}\sd_k^2 + \tfrac{\sL}{10} \|\sz^m-\sx^{m}\|^2 + \tfrac{3\sL}{1-\lambda}\|\sz^n-\sx^n\|^2 + \tfrac{5\scs_k^2}{(1-\lambda)\tw} \\
     &\leq f(\sz^m) + \sL\big[\tfrac{11}{10}+\tfrac{3(1+\lambda)}{2(1-\lambda)}\big]\|\sz^m-\sx^{m}\|^2 + \tfrac{1}{1-\lambda}\big[ \tfrac{10\sL\tw}{1-\lambda} + \tfrac{24\sL\tw}{(1-\lambda)^3}\big]\tw\|\nabla f(\sz^m)\|^2\\
      &\hspace{1.4cm} + \tfrac{1}{1-\lambda}\big[\tfrac{5}{\tw} + \tfrac{10\sL}{1-\lambda} + \tfrac{96\sL}{(1-\lambda)^3}  \big]\scs_k^2\\
    &\leq  f(\sz^m) + \sL\big(\tfrac{3}{1-\lambda} - \tfrac{2}{5}\big)\|\sz^m-\sx^m\|^2
    +\tfrac{17\tw}{25(1-\lambda)}\|\nabla f(\sz^m)\|^2 + \tfrac{8}{(1-\lambda)\tw} \scs_k^2, 
\end{aligned}
\end{equation}
where the second inequality uses \eqref{eq:tedious-01-later}--\eqref{eq:tedious-02-later} for $\sd_k^2$ and $\|\sz^n-\sx^n\|^2$ and the last line follows from $\frac{11}{10} + \frac{3(1+\lambda)}{2(1-\lambda)} = \frac{3}{1-\lambda} - \frac{2}{5}$, $\frac{10\sL\tw}{1-\lambda} + \frac{24\sL\tw}{(1-\lambda)^3} \leq \frac{34\sL\tw}{(1-\lambda)^3} \leq \frac{17}{25}$ and $\frac{10\sL}{1-\lambda} + \frac{96\sL}{(1-\lambda)^3} \leq \frac{106\sL\tw}{(1-\lambda)^3\tw} \leq  \frac{3}{\tw}$.
Rearranging the inequality \eqref{eq:lem:descent-2}, we finally obtain
\[\begin{aligned}
    &\hspace{-2ex} f(\sz^n) + \tfrac{3\sL}{1-\lambda} \|\sz^n-\sx^n\|^2 + \tfrac{\sL}{12}\sd_k^2
    - \tfrac{8}{(1-\lambda)\tw}\scs_k^2\\
    &\leq f(\sz^m) + \sL\big(\tfrac{3}{1-\lambda} - \tfrac{2}{5}\big)\|\sz^m-\sx^m\|^2 - \tfrac{\tw}{50(1-\lambda)}\|\nabla f(\sz^m)\|^2 \\
    &\leq f(\sz^m) + \big[\tfrac{3\sL}{1-\lambda} - \big(\tfrac{2\sL}{5} - \tfrac{6\tw\zeta^2}{50(1-\lambda)} \big) \big]\|\sz^m-\sx^m\|^2 - \tfrac{\tw}{100(1-\lambda)} \|\nabla \mer(\sx^m,\sz^m)\|^2\\
    &\leq f(\sz^m) + \tfrac{3\sL}{1-\lambda}\|\sz^m-\sx^m\|^2 - \tfrac{\tw}{100(1-\lambda)} \|\nabla \mer(\sx^m,\sz^m)\|^2 ,
\end{aligned}\]
where the second inequality follows from the gradient bound \cref{grad-ineq} and the last line is due to $\frac{6\tw\zeta^2}{50(1-\lambda)} = \frac{6\tw}{50(1-\lambda)} (\frac{3\sL}{1-\lambda})^2 \leq \frac{2\sL}{5}$ as $\tw \leq \frac{(1-\lambda)^3}{50\sL}$.
\end{proof}

\subsection{Proof of \texorpdfstring{\Cref{lem:kl-ineq}}{Lemma 4.5}}\label{proof:lem:kl-ineq}

\begin{proof}{Proof}
 Let us fix an arbitrary sample $\omega \in \cS$ and let us set $x^t \equiv \sx^t(\omega)$, $z^t \equiv \sz^t(\omega)$, $s_t \equiv \scs_t(\omega)$, $d_t\equiv \sd_t(\omega)$, etc. By \Cref{lem:approx-desent} and as in \eqref{eq:approximate descent use}, we have
 \begin{equation}\label{eq:lem:KL-1} \mer(x^{\ti_{t+1}},z^{\ti_{t+1}}) + u_{t+1} + \frac{\sL d_t^2}{12} + \frac{\tw\|\nabla \mer(x^{\ti_t}, z^{\ti_t})\|^2}{100(1-\lambda)}  \leq \mer(x^{\ti_{t}},z^{\ti_{t}}) + u_{t},
\end{equation}
for all $t \geq K_\tw$, where $u_t=\frac{8}{(1-\lambda)\tw}{\sum}_{i=t}^\infty s_i^2$. Due to $x^{\ti_k},z^{\ti_k} \in V(x^*)$ and $|\mer(x^{\ti_k},z^{\ti_k}) - f(x^*)| < 1$, \Cref{lem:mer-kl} is applicable and for all $\vartheta\in[\theta,1)$, it holds that 
	\begin{equation}\label{eq:lem:kl-key-3}
		\|\nabla \mer(x^{\ti_k},z^{\ti_k})\| \geq {\sC} |\mer(x^{\ti_k},z^{\ti_k}) - f(x^*)|^\theta \geq \sC |\mer(x^{\ti_k},z^{\ti_k}) - f(x^*)|^\vartheta .
	\end{equation}
	We define $\varrho(x):=\frac{1}{\sC(1-\vartheta)} \cdot x^{1-\vartheta}$ (hence $[\varrho^\prime(x)]^{-1}=\sC x^\vartheta$) and  
 \[\Psi_k := \varrho(\mer(x^{\ti_{k}},z^{\ti_{k}}) - f(x^*) + u_k).\]
 Note that $\Psi_k$ and $\Psi_{k+1}$ are well-defined as $\mer(x^{\ti_{j}},z^{\ti_{j}}) + u_j \geq f(x^*)$, for $j \in \{k,k+1\}$. 
 Based on \eqref{eq:lem:kl-key-3} and $\vartheta\in[\frac12,1)$, we have 
\begin{equation}\label{eq:lem-kl-inv}
     \begin{aligned}
          \nonumber &\hspace{-8ex}[\varrho^\prime(\mer(x^{\ti_{k}},z^{\ti_{k}}) - f(x^*) + u_k)]^{-1} \leq \sC[|\mer(x^{\ti_k},z^{\ti_k}) - f(x^*)| + u_k]^\vartheta\\
          &\leq \sC|\mer(x^{\ti_k},z^{\ti_k}) - f(x^*)|^\vartheta + \sC u_k^\vartheta\leq \|\nabla \mer(x^{\ti_k},z^{\ti_k})\| +\sC u_k^\vartheta,
     \end{aligned}
\end{equation}
where the last inequality follows from the subadditivity of $x\mapsto x^\vartheta$ and \eqref{eq:lem:kl-key-3}. Using the concavity of $\varrho$ (and, w.l.o.g., assuming $\|\nabla \mer(x^{\ti_k},z^{\ti_k})\| +\sC u_k^\vartheta \neq 0$), we obtain
\[\begin{aligned}
    &\hspace{-1ex} \Psi_{k} - \Psi_{k+1} \\ & \geq \varrho^\prime(\mer(x^{\ti_{k}},z^{\ti_{k}}) - f(x^*) + u_k) \left [\mer(x^{\ti_{k}},z^{\ti_{k}}) + u_k - \mer(x^{\ti_{k+1}},z^{\ti_{k+1}}) - u_{k+1} \right] \\ 
     & \geq  \varrho^\prime(\mer(x^{\ti_{k}},z^{\ti_{k}}) - f(x^*) + u_k) \left [\tfrac{\sL}{12}d_k^2 + \tfrac{\tw}{100(1-\lambda)} \|\nabla \mer(x^{\ti_k}, z^{\ti_k})\|^2 \right]  \\
     &\geq \frac{\tfrac{\sL}{12}d_k^2 + \tfrac{\tw}{100(1-\lambda)} \|\nabla \mer(x^{\ti_k}, z^{\ti_k})\|^2}{\|\nabla \mer(x^{\ti_k},z^{\ti_k})\| +\sC u_k^\vartheta} = \frac{\tfrac{(1-\lambda)^3}{600(1+2\mparam)^2}d_k^2 + \tfrac{\tw^2}{100(1-\lambda)}\|\nabla \mer(x^{\ti_k},z^{\ti_k})\|^2}{\tw\|\nabla \mer(x^{\ti_k},z^{\ti_k})\| +\sC \tw u_k^\vartheta}\\[1mm]
    &\geq \frac{(1-\lambda)^3}{200} \cdot \frac{2[d_k/(3(1+2\mparam))]^2 + 2[\tw\|\nabla \mer(x^{\ti_k},z^{\ti_k})\|]^2}{\tw\|\nabla \mer(x^{\ti_k},z^{\ti_k})\| +\sC \tw u_k^\vartheta},
\end{aligned}\]
where the third line is due to \eqref{eq:lem:KL-1} and \eqref{eq:lem-kl-inv} and we use $\sL = \frac{(1-\lambda)^3}{50(1+2\mparam)^2\tw}$. Rearranging the above inequality and using $(a+b)^2 \leq 2a^2 + 2b^2$ gives
\[\begin{aligned}
\Big[\tfrac{d_k}{3(1+2\mparam)}  + \tw\|\nabla \mer(x^{\ti_k},z^{\ti_k})\| \Big]^2 &\leq 2\Big[\tfrac{d_k}{3(1+2\mparam)}\Big]^2 + 2[\tw\|\nabla \mer(x^{\ti_k},z^{\ti_k})\|]^2\\&\leq \tfrac{200}{(1-\lambda)^3}[\Psi_k - \Psi_{k+1}] \cdot [\tw\|\mer(x^{\ti_k},z^{\ti_k})\|  + \sC\tw u_k^\vartheta].
\end{aligned}\]
Taking the square root on both sides of this inequality, it finally follows
\[\begin{aligned}
	\tfrac{1}{3(1+2\mparam)} d_k + \tw\|\mer(x^{\ti_k},z^{\ti_k})\|  & \leq \sqrt{\tfrac{100}{(1-\lambda)^3}[\Psi_k - \Psi_{k+1}] \cdot 2[\tw\|\mer(x^{\ti_k},z^{\ti_k})\|  + \sC\tw u_k^\vartheta] } \\
	& \leq \tfrac{50}{(1-\lambda)^3}[\Psi_k - \Psi_{k+1}] +\tw\|\mer(x^{\ti_k},z^{\ti_k})\|  + \sC\tw u_k^\vartheta, 
\end{aligned}\]
where we applied  $\sqrt{ab} \leq \frac{a}{2}+\frac{b}{2}$ for all $a,b\geq0$. This completes the proof.
\end{proof}

\section{Rates Analysis under General Step Sizes Rates} \label{proof:thm:convergence-rate}
In this section, we provide the proof of \Cref{thm:convergence-rate}.
Setting $\beta_k = g(\tg_k)$ in \eqref{step size-1}, \Cref{lem:err_estimate} implies that $\Prob(\cS)=1$ where $\cS$ is defined in \eqref{step size-1}. Moreover, due to ${\sum}_{k=1}^{\infty}\alpha_k^2  g(\tg_k)^{2} <\infty$ (with $g(x) = x^r$ and $g(x)=\exp(rx)/x^p$), all step size requirements in \Cref{Prop:convergence} and \Cref{thm:finite sum} are satisfied. Hence, we have $\|\sz^k - \sx^k\|\to 0$ a.s$.$ and both  $\{\sx^k\}_k$ and  $\{\sz^k\}_k$ converge to some $\crit(f)$-valued mapping $\sx^*:\Omega \to \crit(f)$ a.s$.$ on $\cX$. Let us define the event $\cR\in\cF$, with $\Prob(\cR)= \Prob(\cX)$, as
\begin{equation}
    \label{def:rate-Event}
    \begin{aligned}
        \cR:=& \;
        \cS\cap\cX\cap \{\omega \in \Omega: \sx^k(\omega)\to \sx^*(\omega)\;\;\text{and}\;\; \sz^k(\omega)\to \sx^*(\omega) \}. 
    \end{aligned}
\end{equation}
Let us fix $\omega\in\cR$ and consider the realizations $x^k\equiv\sx^k(\omega)$, $z^k\equiv\sz^k(\omega)$, $d_k\equiv \sd_k(\omega)$, $s_k\equiv \scs_k(\omega)$, etc.. 
Since $\{x^k\}_k$ and $\{z^k\}_k$ converge to $x^*\in\crit(f)$ and $f$ is continuous, the sequences $\{f(x^k)\}_k$ and $\{f(z^k)\}_k$ converge to $f^*:=f(x^*)$. Additionally, there is $\bar k\geq 1$ such that $x^k,z^k\in V(x^*)$ for all $k\geq \bar  k$. Restating \eqref{eq:lem:KL-1}, it holds that
\begin{equation}\label{eq:thm-rate-4}
  \mer(x^{\ti_{k+1}},z^{\ti_{k+1}}) + u_{k+1}  + \tfrac{\tw}{100(1-\lambda)} \|\nabla \mer(x^{\ti_k}, z^{\ti_k})\|^2 \leq \mer(x^{\ti_{k}},z^{\ti_{k}}) + u_{k},
\end{equation}
for all $k\geq \bar k$, where $u_k=\frac{8}{(1-\lambda)\tw}{\sum}_{i=k}^\infty s_i^2$. Rearranging and applying \Cref{lem:mer-kl} yields
 \[\begin{aligned}
 &\hspace{-6ex} 
 [\mer(x^{\ti_{k}},z^{\ti_{k}}) - f^* + u_{k}] - [\mer(x^{\ti_{k+1}},z^{\ti_{k+1}}) - f^* + u_{k+1}]
 \\ &\geq 
 2\sC_\lambda|\mer(x^{\ti_{k}},z^{\ti_{k}}) - f^*|^{2\theta}, \quad \text{where}\quad \sC_\lambda:= \tfrac{\sC^2 \tw}{200(1-\lambda)}. 
 \end{aligned}\]
Adding $2\sC_\lambda u_k^{2\theta}$ on both sides and using $2(|a|^{2\theta} + |b|^{2\theta}) \geq |a+b|^{2\theta}$, $\theta\in[\frac12,1)$, we get
\begin{equation}\label{eq:thm-rate-y-1}
	y_k - y_{k+1} + 2\sC_\lambda u_k^{2\theta} \geq 
 \sC_\lambda y_k^{2\theta} \quad \text{where}\quad y_k:=\mer(x^{\ti_{k}},z^{\ti_{k}}) - f^* +u_k \geq 0.
\end{equation}
(We have $\{y_k\}_k\subset\R_+$, since $\{y_k\}_k$ decreases monotonically to $0$).\\[1mm] 
\noindent\textbf{Part (a).} \emph{Rates under $g(x)=x^r$, $r>\frac12$.} 
Let us substitute $\beta_k = \tg_{k}^r$. According to \Cref{lem:err_estimate} and using the monotonicity of $\{\beta_k\}_k$, we have
\begin{equation}\label{eq:thm-rate-v-1}
\tg_{\ti_k}^{2r} u_k = \beta_{\ti_k}^2 u_k \leq  \frac{8}{(1-\lambda)\tw}{\sum}_{i=k}^\infty \beta_{\ti_i}^2 s_i^2 \to 0 \quad  \Longrightarrow\quad u_k = o(\tg_{\ti_k}^{-2r}).
\end{equation}
Thus, increasing $\bar k$ if necessary, it follows $2\sC_\lambda u_k^{2\theta} \leq \tg_{\ti_k}^{-4r\theta}$ and \eqref{eq:thm-rate-y-1} reduces to 
\begin{equation}\label{eq:thm-rate-y-2}
    y_{k+1} \leq y_k -  \sC_\lambda y_k^{2\theta} +  \tg_{\ti_k}^{-4r\theta} \quad \forall~k \geq \bar k.
\end{equation}
\noindent\textbf{Step 1:} 
\emph{Rates for $\{u_k\}_k$ and main recursion for $\{y_k\}_k$.} By the definition of $\tg_{k}$ and applying \Cref{lem:time-length}, it holds that  
\begin{equation}\label{eq:thm-rate-v}
   \tg_{\ti_k}^{2r} = \Big({\sum}_{i=1}^{\ti_k}\alpha_i\Big)^{2r} \geq \Big[{\sum}_{i=1}^{\ti_t-1}\alpha_i + \delta\tw (k-t) \Big]^{2r} \geq \sD^{-\frac12} k^{2r},\quad \forall\; k \geq t, 
\end{equation} 
for some $\sD \geq 1$ and for some sufficiently large $t \geq \max\{\bar k,K_\delta\}$, where $K_\delta$ is specified in \Cref{lem:time-length}. W.l.o.g., we assume $ k\geq t$ in the subsequent analysis. Combining \eqref{eq:thm-rate-y-2} and \eqref{eq:thm-rate-v} yields
\begin{equation}\label{eq:thm-rate-y-3}
     y_{k+1} \leq y_k -  \sC_\lambda y_k^{2\theta} +  \sD^{\theta} k^{-4\theta r} \leq  y_k -  \sC_\lambda y_k^{2\theta} +  \sD k^{-4\theta r}.
\end{equation}
In addition, by \eqref{eq:thm-rate-v-1} and \eqref{eq:thm-rate-v} the rate for $\{u_k\}_k$ is readily given by:
\begin{equation}\label{eq:thm-rate-v-2}
   u_k = \cO(k^{-2r}).
\end{equation} 
\textbf{Step 2:} \emph{Rates for $\{y_k\}_k$.} In our derivation, we require the following classical result about the rates of certain sequences, cf$.$ \cite[Lemma 1 and 4]{chung1954stochastic} and \cite[Section 2.2]{pol87}.\begin{lemma} \label{lem:generalized-Chung} Let $\{y_k\}_{k}\subseteq \R_{+}$ and $\beta \geq 0$, $s \in (0,1]$, $t > s$ be given with \[ y_{k+1} \leq \big[1- q\cdot (k+\beta)^{-s} \big] y_k + p\cdot (k+\beta)^{-t}, \quad \text{for some $p, q > 0$}. \]
\begin{enumerate}[label=\textup{(\alph*)},topsep=0pt,itemsep=0ex,partopsep=0ex, leftmargin = 25pt]
		\item If $s=1$ and $t< q+1$, then $ y_k \leq \frac{p}{q+1-t}\cdot (k+\beta)^{1-t} + o((k+\beta)^{1-t})$ as $k \to \infty$. \vspace{1mm}
		\item If $s< 1$, then $y_k \leq \frac{p}{q} \cdot (k+\beta)^{s-t} + o((k+\beta)^{s-t})$ as $k \to \infty$.
	\end{enumerate}
\end{lemma}
We now continue the proof of Step 2. We consider the two cases $\theta=\frac12$ and $\theta \in (\frac12,1)$. \\[1mm]
\noindent\textbf{Case I: $\theta=\frac12$.} The recursion \eqref{eq:thm-rate-y-3} simplifies to 
\[y_{k+1}  \leq (1-\sC_\lambda)y_k +  \sD k^{-2r}. \]
If $\sC_\lambda \geq 1$, then $y_{k+1} \leq \sD k^{-2r}$, implying $y_k = \cO(k^{-2r})$. If $\sC_\lambda < 1$, there exists $K \geq t$ such that $k^{-2r} \leq \frac{2-\sC_\lambda}{2(1-\sC_\lambda)} (k+1)^{-2r}$ for all $k\geq K$. We now make the following claim: for all $k \geq t$ it holds that
\[
y_k \leq \Big[\frac{2-\sC_\lambda}{\sC_\lambda(1-\sC_\lambda)} \cdot \sD + \frac{\max_{t \leq i \leq K}\, y_i}{K^{-2r}}\Big] k^{-2r} =: \sE k^{-2r}. 
\]  
We prove this claim inductively. Obviously, due to $\frac{2-\sC_\lambda}{\sC_\lambda(1-\sC_\lambda)} \geq 0$, it follows $y_k \leq \sE k^{-2r}$ for all $t \leq k \leq K$. Suppose that $y_k \leq \sE k^{-2r}$ holds for some $k > K$. Then, we have 
\[	y_{k+1}  \leq (1-\sC_\lambda) y_k + \sD k^{-2r} \leq [(1-\sC_\lambda) \sE + \sD]  k^{-2r} \leq \sE (k+1)^{-2r},
\]
where the second inequality uses $k^{-2r} \leq \frac{2-\sC_\lambda}{2(1-\sC_\lambda)} (k+1)^{-2r}$. This yields $y_{k} = \cO (k^{-2r})$. \\[1mm]
\noindent\textbf{Case II: $\theta\in(\frac12,1)$.}
Let us define $\mu:=\min\{r,\tfrac{1}{4\theta -2}\}$ and $\sD_\lambda :=[(2\theta-1)(\sC_\lambda \theta)]^{-1}$. Then, we can reformulate \eqref{eq:thm-rate-y-3} as follows
\begin{equation}
    \label{eq:thm-rate-y-4}\begin{aligned}
         y_{k+1} \leq y_k -  \sC_\lambda \big[y_k^{2\theta} -\sD_\lambda^{\frac{2\theta}{2\theta-1}} \cdot k^{-4\theta\mu} \big] +  \sD k^{-4\theta r}.
    \end{aligned}
\end{equation}
Since $\theta>\frac12,$ the function $x \mapsto h_\theta(x) := x^{2\theta}$ is convex on $\R_+$, i.e.,
	\[ h_\theta(y) - h_\theta(x) \geq 2\theta  x^{2\theta-1} (y-x) = 2\theta  x^{2\theta-1}y - 2\theta x^{2\theta} \quad \forall~x,y\in\R_{+}. 
	\]
Rearranging the terms in \eqref{eq:thm-rate-y-4} and using the convexity of $h_\theta$, we have
\[\begin{aligned}
	y_{k+1} &\leq y_k - \sC_\lambda \big[h_\theta(y_k) - h_\theta(\sD_\lambda^{\frac{1}{2\theta-1}} k^{-2\mu})\big] + \sD \cdot  k^{-4\theta r} \\ 
 & \leq \big[1-  2/(2\theta-1) \cdot  k^{-2\mu(2\theta-1)}\big]y_k + \big( 2\theta \sC_\lambda \sD_\lambda^{\frac{2\theta}{2\theta-1}} +  \sD \big) k^{-4\theta \mu},
\end{aligned}\]
where the second inequality utilizes $\mu \leq r$ and $\theta\sC_\lambda\sD_\lambda = (2\theta-1)^{-1}$. Notice that $2\mu(2\theta-1) \leq 1$ and $4\theta\mu - 2\mu(2\theta-1) < 2/(2\theta-1)$. Hence, \Cref{lem:generalized-Chung} is applicable and it follows $y_k = \cO(k^{-2\mu})$. In view of \textbf{Case I \& II}, we conclude that 
\begin{equation}\label{eq:thm-rate-y}
    y_k = \cO(k^{-\psi(\theta)}),\quad \text{where}\quad \psi(\theta):=\min\{2r,\tfrac{1}{2\theta-1}\}.
\end{equation}
\textbf{Step 3:} \emph{Transition to $\|\nabla f(x^k)\|$.}
By \eqref{eq:thm-rate-4} and using $y_k \geq 0$, we can infer
\[
{\tw} \|\nabla \mer(x^{\ti_k}, z^{\ti_k})\|^2\leq  100(1-\lambda)[y_k - y_{k+1}] \leq 100(1-\lambda)y_k= \cO(k^{-\psi(\theta)}).
\]
Thus, combining the gradient bound \cref{grad-ineq} and the $\sL$-smoothness of $f$, we obtain 
\begin{equation}
    \label{eq:thm-rates}
    \max\{\|x^{\ti_k}-z^{\ti_k}\|^2,\|\nabla f(x^{\ti_k})\|^2,\|\nabla f(z^{\ti_k})\|^2\}= \cO(k^{-\psi(\theta)}).
\end{equation}
As shown in \eqref{eq:thm-rate-v-2}, it holds that $s_k^2 = \cO(u_k) = \cO(k^{-2r})$. Hence, in light of \Cref{lem:iterate-bound} and using $\psi(\theta)\leq 2r$, it follows
\begin{equation}
    \label{eq:thm-rate-d}    d_k^2 \leq   \tfrac{3}{2}\|z^{\ti_k} - x^{\ti_k}\|^2 + \tfrac{15}{(1-\lambda)^2}(\tw^2 \|\nabla f(z^{\ti_k})\|^2 + s_{k}^2)  = \cO(k^{-\psi(\theta)}).
\end{equation}
By the definition of $\{d_k\}_k$ (cf.\ \eqref{def:d}) and applying \cref{Assumption:1}, we further have 
\begin{equation}
    \label{eq:thm-rate-grad-1}
    {\max}_{\ell \in \Ti_k} \|\nabla f(x^{\ell})\|^2 \leq 2\|\nabla f(x^{\ti_k})\|^2 + 2\sL^2 d_k^2 = \cO(k^{-\psi(\theta)}).
\end{equation}
Notice that this rate still depends on the time window and time indices $\{\ti_k\}_k$. Hence, our next step is to transfer the obtained results to the original time scale. 
Mimicking the derivations in \eqref{eq:thm-rate-v} and as $\{\alpha_k\}_k$ is non-increasing and $t$ is fixed, we have
%
\begin{equation}
    \label{eq:thm-transition}
    \tg_{\ti_{k+1}} = \tg_{\ti_t-1} + {\sum}_{i=t}^k \tg_{\ti_i,\ti_{i+1}} + \alpha_{\ti_{k+1}} \leq \tg_{\ti_t} + \tw(k+1-t) \leq \bar \sD k 
\end{equation}
for some $\bar\sD > 0$ and all $k$ sufficiently large. Let $\ell \in \N$ be arbitrary. By construction of $\{\ti_k\}_k$, there is $k \geq 1$ such that $\ell\in\Ti_k$. Thus, by \eqref{eq:thm-rate-grad-1}, we can infer $\|\nabla f(x^{\ell})\|^2 = \cO(k^{-\psi(\theta)})$ and, due to $\tg_\ell \leq \tg_{\ti_{k+1}} \leq \bar \sD k$, we finally obtain 
\begin{equation}
    \label{eq:thm-rate-grad-2}
     \|\nabla f(x^{\ell})\|^2 = \cO(\tg_\ell^{-\psi(\theta)})\quad \text{$\forall~\ell$ sufficiently large}.
\end{equation}
\textbf{Step 4:} \emph{Transition to $\{f(x^k)\}_k$.}
Since $y_k=\mer(x^{\ti_{k}},z^{\ti_{k}}) - f^* +u_k=f(z^{\ti_{k}}) - f^* + \zeta\|x^{\ti_{k}}-z^{\ti_{k}}\|^2 +u_k$, by utilizing the triangle inequality, we obtain
\begin{equation}
    \label{eq:thm-rate-func-1}
    |f(z^{\ti_{k}}) - f^*| \leq y_k + \zeta\|x^{\ti_{k}}-z^{\ti_{k}}\|^2 +u_k = \cO(k^{-\psi(\theta)}),
\end{equation}
where the equality holds thanks to \eqref{eq:thm-rate-v-2}, \eqref{eq:thm-rate-y}--\eqref{eq:thm-rates}. Next, we want to show $|f(x^{\ti_{k}}) - f^*|=\cO(k^{-\psi(\theta)})$. Restating \eqref{eq:prop-descent-use-later}, it holds for all $y_1,y_2\in\Rn$ that
\begin{equation}
    \label{eq:thm-descent-use-later}
    |f(y_1) - f(y_2)| \leq \tfrac{1}{2\sL}\max\{\|\nabla f(y_1)\|^2,\|\nabla f(y_2)\|^2\} + \sL\|y_1-y_2\|^2.
\end{equation}
Substituting $y_1=x^{\ti_{k}}$ and $y_2 = z^{\ti_{k}}$ in \eqref{eq:thm-descent-use-later} and using \eqref{eq:thm-rates}, this yields
\[
    |f(x^{\ti_{k}}) - f(z^{\ti_{k}})| \leq \tfrac{1}{2\sL}\max\{\|\nabla f(x^{\ti_{k}})\|^2,\|\nabla f(z^{\ti_{k}})\|^2\} + \sL\|x^{\ti_{k}}-z^{\ti_{k}}\|^2 = \cO(k^{-\psi(\theta)}),
\]
which, along with \eqref{eq:thm-rate-func-1}, further implies 
\begin{equation}
    \label{eq:thm-rate-func-2}
    |f(x^{\ti_{k}}) - f^*| \leq |f(z^{\ti_{k}}) - f^*| +  |f(x^{\ti_{k}}) - f(z^{\ti_{k}})| = \cO(k^{-\psi(\theta)}).
\end{equation}
As discussed before, for every integer $\ell \geq 1$, there is $k \geq 1$ such that $\ell\in\Ti_k$. Replacing $y_1=x^\ell$ and $y_2 = x^{\ti_k}$ in \eqref{eq:thm-descent-use-later}, this yields
\[
  |f(x^{\ell}) - f(x^{\ti_{k}})| \leq \tfrac{1}{2\sL}\max\{\|\nabla f(x^{\ell})\|^2,\|\nabla f(x^{\ti_{k}})\|^2\} + \sL d_k^2 = \cO(k^{-\psi(\theta)}),
\]
where we applied \eqref{eq:thm-rate-d}, \eqref{eq:thm-rate-grad-1}. Combining this with \eqref{eq:thm-rate-func-2} gives $
|f(x^{\ell}) - f^*| \leq |f(x^{\ell}) - f(x^{\ti_{k}})| + |f(x^{\ti_{k}}) - f^*| = \cO(k^{-\psi(\theta)}) $. Using \eqref{eq:thm-transition} and $\tg_\ell \leq \tg_{\ti_{k+1}}$, we have
\[|f(x^{\ell}) - f^*| = \cO(\tg_\ell^{-\psi(\theta)}).\]
\textbf{Step 5:} \emph{Rates for $\{x^k\}_k$.} In this step, we will work with the adjusted {\L}ojasiewicz exponent $\vartheta\in[\theta,1)$ such that $\vartheta > \frac{1}{2r}$ (which ensures the summability of $\{u_k^\vartheta\}_k$). Summing the recursion in \Cref{lem:kl-ineq} from $k=m$ to $n$ leads to
 \[
 {\sum}_{k=m}^n d_k \leq \sM \Psi_m + 3(1+2\mparam)\sC\tw {\sum}_{k=m}^n u_k^\vartheta,
 \]
 where $\Psi_k = \frac{1}{\sC(1-\vartheta)}[\mer(x^{\ti_{k}},z^{\ti_{k}}) - f(x^*) + u_k]^{1-\vartheta} = \frac{1}{\sC(1-\vartheta)} y_k^{1-\vartheta} = \cO(k^{-(1-\vartheta) \psi(\vartheta)})$ and $\sM = \frac{150(1+2\mparam)}{(1-\lambda)^3}$. Letting $n\to\infty$ and noting $u_k= \cO(k^{-2r})$ (by \eqref{eq:thm-rate-v-2}), we have 
 \begin{equation}
    \label{eq:thm-rate-sum-d-1}
 {\sum}_{k=m}^\infty d_k = \cO\Big(m^{-(1-\vartheta)\psi(\vartheta)} + {\sum}_{k=m}^\infty k^{-2r\vartheta} \Big) =  \cO(m^{-(1-\vartheta)\psi(\vartheta)} + m^{1-2r\vartheta}).
 \end{equation}
Since $\vartheta>0$ can be selected freely from $(\tfrac{1}{2r},1)\cap[\theta,1)$, to achieve the best rate given $r>\frac12$ and $\theta\in[\frac12,1)$, we solve the following constrained  optimization problem:
\begin{equation}
    \label{eq:thm-rate-max}
    {\max}_{\vartheta\in[\theta,1)}\; \varphi(\vartheta):=\min\{2r\vartheta-1, 2r(1-\vartheta),\tfrac{1-\vartheta}{2\vartheta-1}\} \quad \text{s.t.}\quad \vartheta > \tfrac{1}{2r}.
\end{equation}
Observe that $\phi(\vartheta) = 2r\vartheta -1$ if $\frac{1}{2r} <\vartheta \leq \frac{1+2r}{4r}$ and $\phi(\vartheta) = \frac{1-\vartheta}{2\vartheta-1}$ if $\frac{1+2r}{4r} < \vartheta < 1$. Clearly, $\varphi(\vartheta)$ increases in $(\frac{1}{2r},\frac{1+2r}{4r})$ and decreases in $(\frac{1+2r}{4r},1)$. Hence, in the case $\theta\geq \frac{1+2r}{4r}$, $\varphi$ attains the maximum at $\theta$; in the case $\frac12 \leq \theta <\frac{1+2r}{4r}$, the maximum is attained at $\frac{1+2r}{4r}$. Consequently, the solution $\vartheta^*$ of \eqref{eq:thm-rate-max} is given by 
\[ \vartheta^*=\begin{cases} \tfrac{1+2r}{4r} & \text{if}\; \frac12 \leq \theta <\frac{1+2r}{4r},\\ \theta  &\text{if}\;\frac{1+2r}{4r} \leq \theta <1, \end{cases}\quad \implies \quad \varphi(\vartheta^*)=\begin{cases} r-\tfrac12 & \text{if}\; \frac12 \leq \theta <\frac{1+2r}{4r}, \\ \tfrac{1-\theta}{2\theta-1} &\text{if}\;\frac{1+2r}{4r} \leq \theta <1,\end{cases} \]
i.e., $\vp(\vartheta^*) = \phi(\theta) := \min\{r-\frac12,\frac{1-\theta}{2\theta-1}\}$, and \eqref{eq:thm-rate-sum-d-1} reduces to ${\sum}_{k=m}^\infty d_k = \cO(m^{-\phi(\theta)})$. 
Recalling $x^k\to x^* \in \crit(f)$, we have $x^{\ti_k} \to x^*$ as $k\to\infty$ and 
\[
\|x^{\ti_m} - x^*\|\leq {\sum}_{k=m}^\infty \|x^{\ti_k} - x^{\ti_{k+1}} \| \leq {\sum}_{k=m}^\infty d_k = \cO(m^{-\phi(\theta)}).
\]
Analogous to our previous steps, for any index $\ell \geq 1,$ there is $m\geq 1$ such that $\ell \in \Ti_m$. Hence, it holds that 
\[
\|x^\ell - x^* \| \leq \|x^\ell - x^{\ti_m}\| + \|x^{\ti_m} - x^*\| \leq d_m + {\sum}_{k=m}^\infty d_k = \cO(m^{-\phi(\theta)}).
\]
Finally, the relation $\tg_\ell \leq \tg_{\ti_{m+1}} \leq \bar \sD m$ (by \eqref{eq:thm-transition}) yields $ \|x^\ell - x^* \| = \cO(\tg_\ell^{-\phi(\theta)})$. \\[1mm]
\noindent\textbf{Part (b).} \emph{Rates under $\theta=\frac12$, $g(x)=\frac{\exp(rx)}{x^p}$ and $r>0$, $p \geq 0$.} Substituting $\theta=\frac12$ in \eqref{eq:thm-rate-y-1} and noting that $g(x)=\frac{\exp(rx)}{x^p}$ and $u_k = o(g(\tg_k)^{-2})$ by \eqref{eq:thm-rate-v-1}, we have
 \begin{equation}\label{eq:p2-thm-recur-y}
    y_{k+1} \leq (1-\sC_\lambda)y_k  +  \tg_{\ti_k}^{2p}\cdot \exp(-2r\tg_{\ti_k}),\quad \text{where}\quad \sC_\lambda=\tfrac{\sC^2 \tw}{200(1-\lambda)}.
\end{equation}
If $\sC_\lambda \geq 1$, we have  $y_{k+1} \leq\tg_{\ti_k}^{2p} \exp(-2r\tg_{\ti_k}) = g(\tg_{\gamma_k})^{-2}$. Next, we consider $\sC_\lambda <1$. 
Since $\alpha_k \to 0$, there is $\tilde k \geq 1$ such that $\alpha_k \leq \tw$ for all $k \geq \tilde k$.
By \Cref{lem:time-length}, we further have for all $k \geq i \geq  \max\{\bar k,\tilde k, K_\delta\}$ that 
\begin{equation}\label{eq:p2-thm-time}
    \tg_{\ti_k} - \tg_{\ti_i} = {\sum}_{j=\ti_i+1}^{\ti_k}\;\alpha_j  \leq \alpha_{\ti_k}+ \tg_{\ti_i,\ti_k} \leq \tw + \tw(k-i).
\end{equation}
Dividing \eqref{eq:p2-thm-recur-y} by $(1-\sC_\lambda)^{k+1}$ gives
$
 \frac{y_{k+1}}{(1-\sC_\lambda)^{k+1}} \leq \frac{y_k}{(1-\sC_\lambda)^k}  +  \frac{\tg_{\ti_k}^{2p}\cdot \exp(-2r\tg_{\ti_k})}{(1-\sC_\lambda)^{k+1}}. 
$
Summing this estimate for $i = t,\dots,k$, $t \geq  \max\{\bar k,\tilde k, K_\delta\}$, and multiplying $(1-\sC_\lambda)^{k+1}$ back on both sides, we obtain 
\[\begin{aligned}
    y_{k+1} &\leq (1-\sC_\lambda)^{k-t+1}   y_t + {\sum}_{i=t}^k \tg_{\ti_i}^{2p} \exp(-2r\tg_{\ti_i})(1-\sC_\lambda)^{k-i}\\
    & \hspace{-0ex}\leq  (1-\sC_\lambda)^{k-t+1}   y_t + \tg_{\ti_k}^{2p} \exp(-2r\tg_{\ti_k}) {\sum}_{i=t}^k  \exp(2r(\tg_{\ti_k}-\tg_{\ti_i})) (1-\sC_\lambda)^{k-i}\\
    &\hspace{-0ex}\leq  (1-\sC_\lambda)^{k-t+1}   y_t + \tg_{\ti_k}^{2p} \exp(-2r\tg_{\ti_k}) {\sum}_{i=t}^k  \exp(2r\tw) [\exp(2r\tw)(1-\sC_\lambda)]^{k-i},
\end{aligned}\]
where the second line uses $\tg_{\ti_k} \geq \tg_{\ti_i}$ and the last line is due to \eqref{eq:p2-thm-time}. The condition $r<\sC^2/400$ implies $\exp(2r\tw)(1-\sC_\lambda)<1$ and hence, we have $(1-\sC_\lambda) < \exp(-2r\tw)$ and ${\sum}_{i=t}^k  \exp(2r\tw) [\exp(2r\tw)(1-\sC_\lambda)]^{k-i} \leq \exp(2r\tw)/[1-\exp(2r\tw)(1-\sC_\lambda)] =:\bar \sC_\lambda$. Thus, due to \eqref{eq:p2-thm-time}, it follows
\[
\begin{aligned}
    y_{k+1} &\leq  \exp(-2r\tw(k-t+1))  y_t  + \bar \sC_\lambda g(\tg_{\gamma_k})^{-2} \\
    &\leq \exp(-2r\tg_{\ti_k}) \cdot \exp(2r\tg_{\ti_t}) y_t  + \bar \sC_\lambda g(\tg_{\gamma_k})^{-2}.
\end{aligned}
\]
Since $\exp(2r\tg_{\ti_t}) y_t$ and $\bar \sC_\lambda$ are fixed, this yields 
$y_{k+1} = \cO(g(\tg_{\gamma_k})^{-2})$. By \eqref{eq:p2-thm-time}, we have $\tg_{\ti_{k+1}} - \tg_{\ti_k} \leq 2\tw$, which implies $\exp(-2r\tg_{\ti_k}) \leq \exp(4r\tw)  \exp(-2r\tg_{\ti_{k+1}})$. Consequently, due to the monotonicity of $\{\tg_{\ti_k}\}_k$, we can infer
\begin{equation}\label{eq:p2-thm-y}
    y_{k+1} = \cO(g(\tg_{\gamma_{k+1}})^{-2})\quad \Longleftrightarrow \quad y_{k} = \cO(g(\tg_{\gamma_k})^{-2}).
\end{equation}
Based on \eqref{eq:p2-thm-y} and $u_k = o(g(\tg_{\gamma_k})^{-2})$, we can obtain the desired rates by fully repeating the procedures used in \textbf{Step 3--5}. We will omit further details. 	
\end{document}